\documentclass[12pt]{amsart}
\usepackage{mathrsfs}
\usepackage{geometry}
\usepackage{titletoc}
\usepackage{mathpazo}
\usepackage{amsmath}
\usepackage{amssymb} 
\usepackage{mathtools} 
\usepackage[table, svgnames]{xcolor} 
\usepackage[all]{xy} 
\usepackage{tikz} 
\usepackage{tikz-cd}
\usepackage{indentfirst} 
\usepackage{setspace} 
\usepackage{enumerate}
\usepackage{color}
\allowdisplaybreaks[4]

\usepackage{hyperref}
\hypersetup{hypertex=true, colorlinks=true, linkcolor=red, anchorcolor=green, citecolor=blue}

\usepackage{rotating} 

\usepackage{ytableau} 
\usepackage{longtable} 
\newcolumntype{M}[1]{>{\centering\arraybackslash}m{#1}} 

\usepackage{mathrsfs}
\DeclareFontFamily{OMS}{rsfs}{\skewchar\font'60}
\DeclareFontShape{OMS}{rsfs}{m}{n}{<-5>rsfs5 <5-7>rsfs7 <7->rsfs10 }{}
\DeclareSymbolFont{rsfs}{OMS}{rsfs}{m}{n}
\DeclareSymbolFontAlphabet{\scr}{rsfs}
\DeclareSymbolFontAlphabet{\scr}{rsfs}

\usepackage[T1]{fontenc}

\newcommand\cE{{\mathcal E}}

\newcommand\cH{{\mathcal H}}

\newcommand{\MA}{\mathrm{MA}}

\newcommand{\ddc}{\mathrm{dd}^c}

\newcommand{\Vol}{\mathrm{Vol}}
\newcommand{\dist}{\mathrm{dist}}

\newcommand{\Psh}{{\textup{Psh}}}

\newcommand{\abs}[1]{\left\lvert#1\right\rvert}
\newcommand{\norm}[1]{\left\lVert#1\right\rVert}

\theoremstyle{theorem}
\newtheorem{theorem}{Theorem}[section]
\newtheorem{proposition}[theorem]{Proposition}

\newtheorem{lemma}[theorem]{Lemma}

\newtheorem{corollary}[theorem]{Corollary}

\theoremstyle{definition}
\newtheorem{definition}[theorem]{Definition}
\newtheorem{remark}[theorem]{Remark}
\newtheorem{example}[theorem]{Example}

\newcommand{\btheorem}{\begin{thm}}
	\newcommand{\etheorem}{\end{thm}}
\newcommand{\bproposition}{\begin{prop}}
	\newcommand{\eproposition}{\end{prop}}
\newcommand{\bdefinition}{\begin{defn}}
	\newcommand{\edefinition}{\end{defn}}
\newcommand{\bcorollary}{\begin{cor}}
	\newcommand{\ecorollary}{\end{cor}}
\newcommand{\bproof}{\begin{proof}}
	\newcommand{\eproof}{\end{proof}}
\newcommand{\bremark}{\begin{remark}}
	\newcommand{\eremark}{\end{remark}}
\newcommand{\eexample}{\end{example}}
\newcommand{\bexample}{\begin{example}}

\newcommand{\elemma}{\end{lemma}}
\newcommand{\blemma}{\begin{lemma}}

\newcommand{\p}{\partial}

\renewcommand{\bar}{\overline}

\newcommand{\beq}{\begin{equation}}
\newcommand{\eeq}{\end{equation}}
\newcommand{\ee}{\end{eqnarray*}}
\newcommand{\be}{\begin{eqnarray*}}

\newcommand{\bd}{\begin{enumerate}}
	\newcommand{\ed}{\end{enumerate}}

\renewcommand{\hat}{\widehat}
\renewcommand{\tilde}{\widetilde}

\renewcommand{\>}{\rightarrow}

\newcommand{\CC}{\mathbb {C}}

\newcommand{\RR}{{\mathbb R}}

\newcommand{\ov}[1]{\overline{#1}}
\newcommand{\de}{\partial}

\newcommand{\ddbar}{\sqrt{-1} \partial \overline{\partial}}

\newcommand{\PSH}{\mathrm{PSH}}
\newcommand{\ve}{\varepsilon}
\newcommand{\e}{\varepsilon}
\newcommand{\vp}{\varphi}

\makeatletter
\newcommand{\constant}[2]{%
  \@ifundefined{c@#1counter}{\newcounter{#1counter}}{}%
  \refstepcounter{#1counter}%
  \protected@write\@auxout{}{%
    \string\newlabel{#2}{{\arabic{#1counter}}{\thepage}{}{#1@\arabic{#1counter}}{}}%
  }%
  #1_{\arabic{#1counter}}%
}
\makeatother

\numberwithin{equation}{section} 

\makeatletter

\ifnum\@ptsize=0  \fi \ifnum\@ptsize=2  \fi

\usepackage{fancyhdr}
\title[Variational]{Some variational problems for the complex Monge--Amp\`ere operator}


\author{Jianchun Chu}
\address{School of Mathematical Sciences, Peking University, Yiheyuan Road 5, Beijing 100871, People's Republic of China}
\email{jianchunchu@math.pku.edu.cn}

\author{Yaxiong Liu}
\address{Department of Mathematics, University of Maryland,
4176 Campus Dr,
College Park, MD 20742, USA}
\email{yxliu238@umd.edu}

\author{Nicholas McCleerey}
\address{Department of Mathematics,
Purdue University, West Lafayette
150 N University St
West Lafayette, IN 47907}
\email{nmccleer@purdue.edu}

\author{Weijun Zhang}
\address{School of Mathematics and Statistics, Nanfang College, Guangzhou, No. 882, Wenquan Avenue, Conghua District, Guangzhou 510970, People's Republic of China}
\email{zwj2017@amss.ac.cn}

\begin{document}
\begin{abstract}
We consider the Dirichlet problem for the complex Monge--Amp\`ere equation on strongly pseudoconvex K\"ahler manifolds when the right-hand side is decreasing in the solution. Using flow-based arguments, we establish existence of smooth solutions in a number of natural circumstances, following work of Chou-Wang.
\end{abstract}
	
	\maketitle
	
\small{\setcounter{tocdepth}{2} \tableofcontents
		\dottedcontents{section}[0.8cm]{}{1.8em}{1555pt}
		\dottedcontents{subsection}[2.0cm]{}{3em}{1555pt}
}

\section{Introduction} \label{Introduction}

Suppose that $(\Omega^n,\omega)$ is a strongly pseudoconvex K\"ahler manifold (see Definition~\ref{def:s-pseudoconvex}), of complex dimension $n$. We are interested in the Dirichlet problem:
\begin{align}\label{eq:initial equation}
	\begin{cases}
		(\ddbar u)^n=\psi^n(z,u)\omega^n, &\text{ in }\Omega, \\
		u=0, &\text{ on }\p\Omega,
	\end{cases}
\end{align}
where $\psi$ is a nonnegative function in $\bar\Omega\times\RR$. We set:
\begin{align*}
\cH(\Omega) \coloneqq\{u\in\Psh(\Omega)\cap C^2(\Omega)\cap C^0(\bar\Omega)~|~u|_{\p\Omega}=0\}
\end{align*}
and write $E(u) \coloneqq \frac{1}{n+1} \int_\Omega (-u)(\ddbar u)^n$ for the Monge--Amp\`ere energy of $u\in \mathcal{H}(\Omega)$. Setting $\Psi(z, x) \coloneqq \int_x^0\psi^n(z,s)ds$, one can check that:
\[
J(u) \coloneqq E(u) - \int_\Omega \Psi(z, u)\omega^n, \quad\quad u\in \mathcal{H}(\Omega)
\]
is an Euler-Lagrange functional for \eqref{eq:initial equation}.\\

It is well-known that \eqref{eq:initial equation} is uniquely solvable if one can prove an {\it a priori} $L^\infty$ bound on the solution c.f. \cite{PSS12}. This is in particular always possible when $\psi$ is increasing in the second variable, by a simple application of the maximum prinicple. When $\psi$ is decreasing in the second variable however, this is no longer possible, and indeed, both existence and uniqueness of solutions may fail.

An important special case is the {\bf eigenvalue problem}, where $\psi(z, u) = -\lambda u(z)$, for a constant $\lambda > 0$ which is also to be solved for; recently, this eigenvalue problem was studied by Badiane-Zeriahi \cite{BZ23}, following the method of Lions \cite{Lions85}. They establish the existence of a unique $\lambda$, known as the {\bf first eigenvalue} of $(\Omega, \omega)$, for which one can solve \eqref{eq:initial equation}. The solution itself, known as the {\bf first eigenfunction}, is unique up to a multiplicative constant. From now on, we denote this pair $(u_1, \lambda_1)$, with the normalization $\inf_\Omega u_1 = -1$. We refer to \cite{BZ24, CLM26, LZ25} for follow-up works on related eigenvalue problems in the complex setting. \\

The first eigenvalue turns out to act as a natural bifurcation parameter for \eqref{eq:initial equation} for general $\psi$ -- for instance, following \cite[Corollary 2]{Lions85}, it was shown in \cite[Theorem 1.3]{CLM26} that \eqref{eq:initial equation} always admits a unique solution when $\psi \geq \e > 0$ and $\p_u \psi \geq -\lambda_1 + \e$, for some $\e > 0$.

In this paper, we aim to study existence of solutions to \eqref{eq:initial equation} under weaker, asymptotic assumptions on the growth of $\psi$ in $u$. For the real Monge--Amp\`ere equation (and also the real $k$-Hessian equations), this was done by Chou--Wang \cite{CW01}, who develop results similar to those for the semi-linear equation:
\[
\Delta u = \psi(x, u)
\]
studied by Ambrosetti--Rabinowitz \cite{AR73}. We largely follow the methods of \cite{CW01}, with adaptations to the complex setting; the main difficulty lies in developing new effective {\it a priori} estimates for certain flows related to \eqref{eq:initial equation}. \\

Our first result extends the Monge--Amp\`ere case of \cite[Theorem 1.3]{CLM26}, and concerns $\psi$ which grow ``sublinearly" in the solution variable:
\begin{theorem}[Sublinear Growth]\label{Sublinear Theorem}
 Let $\psi\in C(\bar{\Omega\times\RR_-})\cap C^{1,1}(\bar\Omega\times\RR_-)$. Suppose that $\psi(z,x)>0$ for $x<0$, and that
	\begin{align}
	&\lim_{x\>0^-}\frac{\psi(z,x)}{|x|} > \lambda_1 \quad \text{uniformly for } z\in \ov{\Omega}, \text{ and } \label{eq:sublinear growth cond go to 0- intro}\\
	&\lim_{x\rightarrow-\infty}\frac{\psi(z,x)}{|x|} <\lambda_1 \quad \text{uniformly for } z\in\ov{\Omega}. \label{eq:sublinear_growthgoto-infty}
	 \end{align}
 Then there exists a non-zero solution $u\in  C^{1,1}(\bar\Omega)\cap C^\infty(\Omega) \cap \cH(\Omega)$ to \eqref{eq:initial equation} which minimizes $J$ over $\mathcal{H}(\Omega)$.
\end{theorem}

As mentioned, our approach to show Theorem \ref{Sublinear Theorem} follows \cite[Theorem 1.3]{CW01}. Briefly, the Rayleigh quotient formula for $\lambda_1$ \cite[Theorem 1.2]{BZ23} (recalled in Section 2 below) and the sublinear growth condition \eqref{eq:sublinear_growthgoto-infty} ensure that $J$ is proper on $\mathcal{H}(\Omega)$ (with respect to the Monge--Amp\`ere energy). By combining this with an effective {\it a priori} H\"older bound, one can deduce a uniform bound, see Remark 4.4. By standard theory, c.f. Theorem \ref{thm:C3alpha-estimate}, it follows that \eqref{eq:initial equation} admits a solution. To produce an actual minimizer however, we need to employ a natural (negative) gradient flow for $J$, which requires a short approximation argument.

Note that, although $J$ is proper, it will not typically have any good convexity properties; as such, solutions to \eqref{eq:initial equation} may not be unique in general. The asymptotic condition \eqref{eq:sublinear growth cond go to 0- intro} ensures that any minimizer is non-zero and is, in some sense, sharp; e.g. for smooth, decreasing, concave $\psi(z,x) = \psi(x)$, $x\leq 0$, non-trivial solutions fail to exist precisely when \eqref{eq:sublinear growth cond go to 0- intro} does not hold, see \cite[Proposition 3]{Lions85}. \\

Our next result considers ``superlinear" growth conditions on $\psi$, based off of the original conditions in \cite{AR73, CW01}. In this case, $J$ will no longer be proper, and we must look for saddle point solutions. We do so by employing an argument similar to the mountain-pass lemma, following again Chou--Wang \cite[Theorem 1.2]{CW01}. 

\begin{theorem}[Superlinear]\label{superlinear_theorem_intro}
	Let $\psi\in C(\bar{\Omega\times\RR_-})\cap C^{1,1}(\bar\Omega\times\RR_-)$. Suppose that $\psi(z, x) > 0$ for $z < 0$, and that there exists constants $0 < \sigma < 2n$, $0 < \theta < 1$, and $M > 0$ such that:
\begin{align}
&\lim_{x\>0^-}\frac{\psi(z,x)}{|x|}<\lambda_1 \quad \text{uniformly for } z\in \ov{\Omega} \label{eq:growth cond go to 0- intro}\\
&\lim_{x\rightarrow -\infty} \psi(z, x)e^{-\frac{\sigma}{n} \abs{x}^{1+\frac{1}{n}}} <\infty \quad \text{uniformly for }z\in\ov{\Omega}\label{eq:exp growth intro}\\
&\int_x^0\psi^n(z,s)ds \leq\frac{1-\theta}{n+1}|x|\psi^n(z,x) \quad \text{ for all $x \leq -M$ and $z\in\ov{\Omega}$.} \label{eq:integration growth upper intro}
\end{align}
Then there exists a non-zero solution $u\in C^{1,1}(\bar\Omega) \cap C^{\infty}(\Omega) \cap \cH(\Omega)$ of \eqref{eq:initial equation}.
\end{theorem}

While there are many results studying convex Euler--Lagrange functionals for the complex Monge--Amp\`ere equation, this seems to be first that we are aware of which explicitly studies saddle point solutions.\\

Let us briefly discuss the assumptions in Theorem \ref{superlinear_theorem_intro}, in reverse order. By using the integral form of Gr\"onwall's Inequality, condition \eqref{eq:integration growth upper intro} implies that:
\begin{equation}\label{eq:superlinear condition intro}
\lim_{x\>-\infty}\frac{\psi(z,x)}{|x|} = \infty \quad \text{uniformly for } z\in \ov{\Omega}.
\end{equation}
In this sense, $\psi$ exhibits ``superlinear" growth, as discussed in the introduction to \cite{CW01}.

We turn to condition \eqref{eq:exp growth intro}. For the real Monge--Amp\`ere operator \cite[Theorem 1.2 when $k=n$]{CW01}, it is not necessary to restrict the growth of $\psi$; this follows from the Hessian--Sobolev inequalities in \cite{CW01}, which bound the $L^\infty$ norm of a $k$-convex function in terms of the energy. For the complex Monge--Amp\`ere operator, such a strong inequality is not possible; the best is the sharp Moser--Trudinger inequality \cite{ WWZ20,DNGL21,BB22}. This is directly analogous to the case when $k = \frac{n}{2}$ in \cite{CW01}; indeed, the Moser--Trudinger inequality is one of several well-known similarities between the complex Monge--Amp\`ere operator and the real $\frac{n}{2}$-Hessian operator, see e.g. \cite[Page 1737]{Blo05} for similar remarks. As such, it is most natural to compare our Theorem \ref{superlinear_theorem_intro} with the case of $k = \frac{n}{2}$ in \cite[Theorem 1.2]{CW01}, rather than the case of $k = n$. The condition in \cite{CW01} corresponding to \eqref{eq:exp growth intro} is then \cite[(1.5)]{CW01}, which is
\[
\lim_{x\>-\infty}\frac{\psi(z,x)}{|x|^{p}}=0 \quad \text{uniformly for } z\in \ov{\Omega}
\]
for some large $p>0$. Compared to this, \eqref{eq:exp growth intro} is weaker, and seems more natural.

Lastly, condition \eqref{eq:growth cond go to 0- intro} is again used to ensure the found solution is non-zero; compared to the sublinear case, sharpness of this condition is less clear. Note that it has the side effect of making \eqref{eq:initial equation} degenerate at the boundary of $\Omega$. Loosening this condition would likely require imposing stronger bounds on the growth of $\psi$ -- for instance, it is known that the volume-normalized K\"ahler--Einstein equation:
\begin{equation*}
(\ddbar u)^n = \frac{e^{-\gamma u}\omega^n}{\int_\Omega e^{-\gamma u} \omega^n}
\end{equation*}
can only be solved for $\gamma < 2n$ \cite{BB22, GKY13}. Not coincidentally, this upper-bound is the same as the optimal constant in the Moser--Trudinger inequality.

Beyond the above technical questions, it would also be interesting to extend our results to singular spaces; by using Hironaka's theorem, this amounts to allowing $\psi$ to blow-up at certain points in $\Omega$. We mention the recent work \cite{HK25}, which studies a case of this for the real $k$-Hessian equations.\\

The rest of this paper is organized as follows. In Section \ref{Preliminary Section}, we collect some preliminary facts we will need in the later sections. In Section \ref{mu flow section}, we prove a priori estimates for the $\mu$-flow, with some additional details given in Appendix \ref{mu appendix}. In Section \ref{sublinear section}, we prove Theorem \ref{Sublinear Theorem}, and in Section \ref{superlinear section}, we prove Theorem \ref{superlinear_theorem_intro}.\\

{\noindent \bf Acknowledgments:} The authors would like to thank Mathew George for related discussions. The first-named author was partially supported by National Key R\&D Program of China 2024YFA1014800 and 2023YFA1009900, and NSFC grants 12471052 and 12271008.

\section{Preliminaries}\label{Preliminary Section}

Let $(\Omega^n,\omega)$ be a complex manifold with boundary, of complex dimension $n$, such that $\ov{\Omega}$ is compact and $\omega$ extends as a K\"ahler metric to a neighborhood of any point in $\p\Omega$ (when considered in a boundary coordinate chart, say). For instance, $\Omega$ may be an open (connected, with smooth boundary) relatively compact subset of a compact K\"ahler manifold without boundary.

\begin{definition}
\label{def:s-pseudoconvex}
We say that $\Omega$ is a \textit{strongly pseudoconvex manifold} if there exists a function $\rho\in \PSH(\Omega)\cap C^\infty(\ov{\Omega})$ such that
\begin{itemize}\setlength{\itemsep}{1mm}
\item $\rho<0$ in $\Omega$ and $\rho=0$ on $\p\Omega$;
\item $d\rho\not= 0$ on $\p\Omega$; and
\item $\ddbar \rho \geq\e_{0}\omega$ for some constant $\e_{0}> 0$.
\end{itemize}
\end{definition}

\begin{definition}
For $u\in\PSH(\Omega)\cap C^{2}(\ov{\Omega})$, we define the (complex) Monge--Amp\`ere operator to be
\[
\MA(u) = \frac{(\ddbar u)^{n}}{\omega^{n}}.
\]
\end{definition}

The eigenvalue problem for the Monge--Amp\`ere operator on a strong psuedoconvex domain in $\CC^n$ was discussed in \cite{BZ23}, where it was shown that there exists a unique pair $(u_1, \lambda_1) \in \mathcal{H}(\Omega) \times \RR_+$ such that:
\begin{equation}\label{MA eval}
\begin{cases}
\MA(u_1) = (-\lambda_1 u_1)^n &\text{ in }\Omega\\
u_1 = 0 &\text{ on }\p\Omega\\
\inf_\Omega u_1 = -1.
\end{cases}
\end{equation}
Actually, it was shown that $u_1 \in C^{1,1-\e}(\ov{\Omega})$ for any $\e > 0$ sufficiently small. In \cite{CLM26}, this was improved slightly to allow $\Omega$ to be strongly pseudoconvex manifold, and it was shown that $u_1 \in C^{1,1}(\ov{\Omega})$. Recently, \cite{LZ25} have extended uniqueness of $u_1$ to hold in the finite energy class.

We refer to the unique pair $(u_1, \lambda_1)$ solving \eqref{MA eval} as the {\it first eigenfunction} and {\it first eigenvalue} of the Monge--Amp\`ere operator on $(\Omega, \omega)$.\\

Building on work of Chou \cite{Tso90}, the following Rayleigh quotient formula for $\lambda_1$ was established in \cite[Theorem 1.2]{BZ23} (their proof applies equally well to strongly pseudoconvex manifolds, essentially without change):
\begin{theorem}[Badiane--Zeriahi \cite{BZ23}]
We have:
\begin{equation}\label{eq:Rayleigh Quotient}
\lambda_1^n = \inf \left\{ \frac{E(u)}{I(u)}\ \middle|\ u\in\mathcal{H}(\Omega) \right\},
\end{equation}
where $E(u) = \frac{1}{n+1}\int_\Omega (-u)\MA(u) \omega^n$ is the Monge--Amp\`ere energy of $u$ and $I(u) := \frac{1}{n+1}\int_\Omega (-u)^{n+1}\omega^n$.
\end{theorem}
It is clear from \eqref{MA eval} that $E(u_1) = \lambda_1^n I(u_1)$, so the infimum is always attained. Badiane--Zeriahi actually show that the result holds if the infimum is taken over all $u$ in the finite energy class, but we will not need this here.\\

We will also need the following strong rigidity, given essentially in \cite[Lemma 3.4]{LZ25}:
\begin{proposition}\label{supersolution prop}
Suppose that $u\in\mathcal{H}(\Omega)$ satisfies $\MA(u) \leq (-\mu(z) u)^n$ for a non-negative function $\mu$ with $\sup_\Omega \mu < \lambda_1$. Then $u = 0$.
\end{proposition}
\begin{proof}
Integrating and applying the Rayleigh quotient formula gives:
\[
E(u) \leq (\sup_\Omega \mu)^n I(u) \leq \lambda_1^n I(u) \leq E(u);
\]
since $\sup_\Omega \mu^n < \lambda_1^n$, we must have $E(u) = I(u) = 0$.
\end{proof}

\medskip

An important tool we will need is the optimal Moser--Trudinger Inequality, originally shown in \cite{BB22}. There are now several variants on this inequality in slightly different contexts, c.f. \cite{GKY13, BB14, AC19, Ceg19, WWZ20, DMV23, DNGL21, GT25} for a non-exhaustive list.
\begin{proposition}(\textup{Moser--Trudinger inequality})\label{prop:MT_inequ}
	Let $0<\gamma<2n$. Then there exists a $C_\gamma>0$ such that
\begin{align}
\label{eq:MT_ineq}
	\int_\Omega \exp\left(\gamma \frac{|u|^{1+\frac{1}{n}}}{E(u)^{\frac{1}{n}}}\right)\omega^n
	\leq C_\gamma,
\end{align}
for all non-constant $u\in\cE^1(\Omega)$, the finite energy class.
\end{proposition}
\begin{proof}
We briefly recall (a special case of) the argument of Di Nezza--Guedj--Lu \cite[Theorem 4.6]{DNGL21}. By approximation, it suffices to assume $u\in \mathcal{H}(\Omega)$ with $\int_\Omega \MA(u) < \infty$.
By using a finite cover, one can apply the local estimate in \cite[Proposition 6.1]{ACKPZ09}, as done in \cite[Theorem 4.6]{DNGL21}, to see that for any $\beta \in (\gamma, 2n)$, there exists a constant $C = C(\Omega,\omega,\beta) > 0$ such that:
\[
\mathrm{Vol}(\{u < -t\}) \leq C\, \mathrm{exp}\left(-\frac{\beta t^{1+\frac{1}{n}}}{E(u)^{\frac{1}{n}}}\right),
\]
for any $t > 0$.

It then follows that:
\[
\int_\Omega\mathrm{exp}\left(\gamma\frac{\abs{u}^{1+\frac{1}{n}}}{E(u)^{\frac{1}{n}}}\right) \omega^n = \int_0^\infty \left(1+\frac{1}{n}\right) \gamma t^{\frac{1}{n}} \mathrm{exp}\left(\gamma t^{1+\frac{1}{n}}\right) \mathrm{Vol}\left(\left\{\frac{u}{E(u)^{\frac{1}{n+1}}} < -t\right\}\right)\,dt
\]
\[
\leq C \int_0^\infty t^{\frac{1}{n}} \mathrm{exp}\left(\left(\gamma - \beta\right) t^{1+\frac{1}{n}}\right)\, dt < \infty,
\]
since $E(cu) = c^{n+1} E(u)$ for any $c > 0$.
\end{proof}

We frequently make use of the following corollary:

\begin{corollary}\label{cor:hessianSob}
    For any $p>1$ and $\gamma \in (0, 2n)$, there exists a constant $C > 0$, depending only on $n, p$, and $\gamma$, such that
    \begin{align}    \label{eq:Sobolev ineq}
        \|u\|_{L^p(\Omega,\omega^n)}  \leq C E(u)^{\frac{1}{n+1}},
    \end{align}
    for all $u\in\cH(\Omega)$.
\end{corollary}
\begin{proof}
By elementary calculus, there exists $C > 0$ depending only on $n, p,$ and $\gamma$, such that
\[
x^p \leq Ce^{\gamma x^{1+\frac{1}{n}}}\text{ for all }x \geq 0.
\]
It follows that:
\[
\int_\Omega \frac{\abs{u}^p}{E(u)^{\frac{p}{n+1}}}\omega^n \leq C \int_\Omega \mathrm{exp}\left(\gamma \frac{\abs{u}^{1+\frac{1}{n}}}{E(u)^{\frac{1}{n}}}\right)\omega^n;
\]
the proposition follows immediately from applying Proposition~\ref{prop:MT_inequ}.\qedhere
\end{proof}

We conclude this section by recalling the standard existence and regularity result for the Dirichlet problem:

\begin{theorem}[\textup{\cite{PSS12}}]
\label{thm:C3alpha-estimate}
	 Suppose that $\psi\in C^{1,1}(\bar{\Omega\times\RR_-})$, with $\psi(z,x)\geq\psi_0>0$, and that $u \in \mathcal{H}(\Omega)$ is a solution to \eqref{eq:initial equation}. Then there exists a constant $C_1>0$, depending only on $n, \Omega$, $M_0\coloneqq\sup_\Omega|u|$, and $\|\psi\|_{C^{1,1}(\bar\Omega\times[-M_0,0])}$, such that
	\begin{align*}
		\|u\|_{C^{2}(\bar\Omega)}
		\leq C_1.
	\end{align*}
    In particular, $C_1$ is independent of $\psi_0$. Furthermore, there exists a constant $C_2$, depending only on $\psi_0$ and the quantities above, such that
    \begin{align*}
        \|u\|_{C^{3,\alpha}(\bar\Omega)} \leq C_2.
    \end{align*}
\end{theorem}

\section{The $\mu$-flow}\label{mu flow section}

In this section, we investigate certain flows for the complex Monge--Amp\`ere equation which will be needed for the later sections. We prove {\it a priori} estimates for solutions, which imply the long time existence by standard parabolic theory when the $C^0$-norm is under control, and prove a stability result.\\

We set notation. Let $(\Omega,\omega)$ again be a strongly pesudoconvex $n$-dimensional K\"ahler manifold with boundary $\p \Omega$. For any $T>0$, we introduce the cylindrical domains:
\[
Q_T = \Omega\times(0,T), \ \
\Gamma_{T} = \de\Omega\times(0,T), \ \
\Omega_{t}=\{(z,t)\in Q_{T}\}.
\]
Then the parabolic boundary of $Q_{T}$ can be written as
\[
\de^{*}Q_{T} = \de Q_{T}-\Omega_{T} = \Omega_{0} \cup \bar{\Gamma}_{T}.
\]
Fix $p \in (n+1, \infty]$. If $p < \infty$, then by Lemma \ref{construction of mu}, there exists a smooth one-variable function $\mu:(0,\infty)\to\mathbb{R}$ satisfying
\begin{itemize}\setlength{\itemsep}{1mm}
\item $\mu(t)=\log t$ for $0<t\leq1$;
\item $\mu(t)=t^{1/p}$ for $t\geq e^{2}$;
\item $\mu'(t)>0$;
\item $t^{2}\mu''(t)+t\mu'(t) \leq \frac{t}{p}\mu'(t)$.
\end{itemize}
If $p = \infty$, we set $\mu(t) = \log(t)$. Set $F(u)=\mu\big(\MA(u)\big)$ and consider the parabolic Dirichlet problem
\begin{equation}\label{parabolic Dirichlet problem Q T}
\begin{cases}
\ F(u)-u_{t} = f(z,t,u) & \mbox{in $Q_{T}$},  \\[1mm]
\ u = u_{0} & \mbox{on $\bar{\Omega}_{0}$}, \\[1mm]
\ u = 0 & \mbox{on $\bar{\Gamma}_{T}$},
\end{cases}
\end{equation}
where
\begin{itemize}\setlength{\itemsep}{1mm}
\item $f\in C^{\infty}(\bar{\Omega}\times\mathbb{R}\times \mathbb{R})$ satisfies
\[
|f|\leq K_{1}+K_{2}|u|, \ \ |f_{t}|+|f_{u}|+|\nabla f| \leq K_{3}
\]
for some constants $K_{1},K_{2},K_{3}\geq0$;
\item $u_{0}\in\PSH(\Omega)\cap C^\infty(\ov{\Omega})$ with $u_{0}=0$ on $\de\Omega$ and $\ddbar u_{0}\geq\delta_{0}\omega$ for some constant $\delta_{0}>0$;
\item $u(\cdot,t)$ is plurisubharmonic in $\Omega_{t}$ for each $t\in[0,T)$.
\end{itemize}

Our main result in this section is then:
\begin{theorem} \label{thm:estimate of flow}
Suppose further that the compatibility condition:
\begin{equation}\label{compatibility condition}
F(u_{0}) = f(z,0,u_{0}) \ \ \text{in $\de\Omega\times\{0\}$}
\end{equation}
holds. Then there exists a smooth solution $u$ of \eqref{parabolic Dirichlet problem Q T} such that for any positive even integer $k$,
\[
\|u\|_{C^{k,k/2}} \leq C,
\]
where $C$ is a constant depending only on $T$, $k$, $u_{0}$, $f$, $(\Omega,\omega)$. In particular, if $f$ is independent of $t$ and $K_{2}=0$, then the constant $C$ can be chosen independently of $T$.
\end{theorem}

For the remainder of this section, we assume $u(z,t)$ to be a smooth solution to \eqref{parabolic Dirichlet problem Q T}, unless specified otherwise; the next several subsections are devoted to proving {\it a priori} estimates for $u$. The compatibility condition \eqref{compatibility condition} will only be needed for the parabolic Evans-Krylov arguments at the end of arguments. We conclude with a final subsection on stability.

\subsection{Uniform estimate} We begin with an $L^\infty$-estimate for $u$.
\begin{proposition}\label{uniform estimate}
Let $u$ be a smooth solution of \eqref{parabolic Dirichlet problem Q T}. If $K_{2}>0$, then
\[
-\left(\|u_{0}\|_{C^{0}}+\frac{K_{1}}{K_{2}}\right)e^{K_{2}t} \leq u \leq 0 \ \ \text{in $\bar{Q}_{T}$}.
\]
If $K_{2}=0$, then there exists a constant $C$ depending only on $K_{1}$, $\mu$ and $(\Omega,\omega)$ such that
\[
-\|u_{0}\|_{C^{0}}-C \leq u \leq 0 \ \ \text{in $\bar{Q}_{T}$}.
\]
\end{proposition}

\begin{proof}
For any fixed $t>0$, $u(\cdot,t)$ is plurisubharmonic in $\Omega_{t}$. By the maximum principle, we have
\[
\sup_{\Omega_{t}}u(\cdot,t) \leq \sup_{\p\Omega_{t}}u(\cdot,t) = 0.
\]
It follows that $u \leq 0$ in $Q_{T}$. It thus suffices to establish a lower bound for $u$.\\

{\noindent \bf Case 1:} $K_{2}>0$. Set $v=e^{-K_{2}t}u$ and compute
\[
\begin{split}
e^{-K_{2}t}F(e^{K_{2}t}v)-v_{t}
= {} & e^{-K_{2}t}F(u)-e^{-K_{2}t}u_{t}+K_{2}e^{-K_{2}t}u \\
= {} & e^{-K_{2}t}f+K_{2}e^{-K_{2}t}u \\
\leq {} & e^{-K_{2}t}(K_{1}-K_{2}u)+K_{2}e^{-K_{2}t}u \\
= {} & K_{1}e^{-K_{2}t}.
\end{split}
\]
Define
\[
w = -\frac{K_{1}}{K_{2}}\left(1-e^{-K_{2}t}\right)-\|u_{0}\|_{C^{0}}.
\]
It follows that
\[
e^{-K_{2}t}F(e^{K_{2}t}w)-w_{t} = K_{1}e^{-K_{2}t}
\]
and so
\[
\begin{cases}
\ e^{-K_{2}t}F(e^{K_{2}t}v)-v_{t} \leq e^{-K_{2}t}F(e^{K_{2}t}w)-w_{t} & \mbox{in $Q_{T}$},  \\[1mm]
\ v \geq w & \mbox{on $\p^{*}Q_{T}$}.
\end{cases}
\]
Since the operator $f \mapsto e^{-K_2t}F(e^{K_2 t} f)$ is still elliptic for $f\in\mathcal{H}(\Omega)$, the parabolic comparison principle now shows that $v\geq w$ in $Q_{T}$, which implies
\[
u \geq -\frac{K_{1}}{K_{2}}\left(e^{K_{2}t}-1\right)-e^{K_{2}t}\cdot \|u_{0}\|_{C^{0}} \geq -\left(\|u_{0}\|_{C^{0}}+\frac{K_{1}}{K_{2}}\right)e^{K_{2}t}.
\]

\medskip
{\noindent \bf Case 2: $K_{2}=0$}. In this case, we have $|f|\leq K_{1}$. Since $\Omega$ is a strongly pseudoconvex manifold, then the defining function $\rho$ satisfies $\rho\leq0$ and $\ddbar\rho\geq\e_{0}\omega$ for some $\ve_{0}>0$. Set $\underline{u}=u_{0}+A\rho$. By $\ddbar u_{0}\geq0$, we have
\[
\MA(\underline{u}) = \MA(u_{0}+{A}\rho) \geq \MA({A}\rho) \geq {A}^{n}\ve_{0}^{n},
\]
and so
\[
\begin{cases}
\ F(\underline{u})-\underline{u}_{t} \geq \mu({A}^{n}\ve_{0}^{n}) & \mbox{in $Q_{T}$},  \\[1mm]
\ \underline{u} = u_{0}+{A}\rho \leq u_{0} & \mbox{on $\Omega_{0}$}, \\[1mm]
\ \underline{u} = 0 & \mbox{on $\bar{\Gamma}_{T}$}.
\end{cases}
\]
Using $\lim_{t\to\infty}\mu(t)=\infty$, we choose ${A}$ sufficiently large such that
\[
\mu({A}^{n}\ve_{0}^{n}) \geq K_{1} \geq f = F(u)-u_{t}.
\]
It follows from the comparison principle that
\[
u \geq \underline{u} \geq -\norm{u_0}_{C^0} - C \ \ \text{in $Q_{T}$}.\qedhere
\]
\end{proof}

\subsection{Estimate for $u_{t}$}
\begin{proposition}\label{u t estimate}
Let $u$ be a smooth solution of \eqref{parabolic Dirichlet problem Q T}. Then there exists a constant $C$ depending only on $u_{0}$, $K_{1}$, $K_{2}$, $K_{3}$, $\mu$ and $(\Omega,\omega)$ such that
\[
\sup_{\bar{Q}_{T}}|u_{t}| \leq CM,
\]
where $M=\|u\|_{C^{0}(\ov{Q}_T)}+1$.
\end{proposition}

\begin{proof}

We first establish the lower bound $u_{t}\geq-CM$. Consider the quantity
\[
G = \frac{u_{t}}{2M-u}.
\]
Let $p_{0}=(z_{0},t_{0})$ be the minimum point of $G$ in $\bar{Q}_{T}$; it will suffice to show $u_{t}(p_{0})\geq-CM$. If $t_0 = 0$, then we have $\ddbar u_0 \geq \delta_0\ddbar\omega$, so that
\[
u_t(p_0) = F(u(p_0)) - f(z_0, t_0, u(p_0))  \geq \mu(\delta_0^n) - K_1-K_2M.
\]
If $p_0 \in \ov{\Gamma}_T$,  then clearly $u_t = 0$. We thus consider when $z_{0}$ is an interior point of $\Omega_{t_{0}}$. Choose a holomorphic normal coordinate system $(U,\{z_{i}\}_{i=1}^{n})$ centered at $z_{0}$. Then at $p_{0}$, we have
\begin{equation}\label{t derivative equation}
0 \geq G_{t} = \frac{u_{tt}}{2M-u}+\frac{u_{t}^{2}}{(2M-u)^{2}},
\end{equation}
\[
0 = G_{i} = \frac{u_{ti}}{2M-u}+\frac{u_{t}u_{i}}{(2M-u)^{2}},\quad\quad 0 = G_{\ov{j}} = \frac{u_{t\ov{j}}}{2M-u}+\frac{u_{t}u_{\ov{j}}}{(2M-u)^{2}},
\]
and
\[
\begin{split}
0 \leq (G_{i\bar j})
= {} & \left(\frac{u_{ti\bar j}}{2M-u}+\frac{u_{ti}u_{\bar j}+u_{t\bar j}u_i+u_tu_{i\bar j}}{(2M-u)^2}
+\frac{2u_tu_iu_{\bar j}}{(2M-u)^3} \right) \\
= {} & \left(\frac{u_{ti\bar j}}{2M-u}+\frac{u_tu_{i\bar j}}{(2M-u)^2}
+\frac{G_{i}u_{\ov{j}}+u_{i}G_{\ov{j}}}{(2M-u)} \right) \\
= {} & \left(\frac{u_{ti\bar j}}{2M-u}+\frac{u_tu_{i\bar j}}{(2M-u)^2}\right).
\end{split}
\]
By rearranging \eqref{t derivative equation} and taking the trace of $(G_{i\ov{j}})$, we see:
\begin{equation}\label{u t estimate eqn 1}
u_t^2 \leq -u_{tt}(2M-u), \quad\text{ and }\quad
-u^{i\ov{j}}u_{ti\ov{j}}(2M-u) \leq u^{i\ov{j}}u_{i\ov{j}}u_{t} = nu_{t}.
\end{equation}
Without loss of generality, we assume that $\MA(u)<1$ at $p_{0}$ (otherwise $\MA(u)\geq1$ implies $F(u)\geq0$ and so $u_{t}=F(u)-f\geq-K_{1}-K_{2}M$, as required). Then $F(u)=\log\MA(u)$ near $p_{0}$, and differentiating the flow \eqref{parabolic Dirichlet problem Q T} with respect to $t$ gives
\[
u^{i\bar j}u_{i\bar jt}-u_{tt} = f_t+f_uu_t.
\]
Combining this with \eqref{u t estimate eqn 1},
\[
\begin{split}
u_t^2 \leq {} & -u_{tt}(2M-u) \\
= {} & - u^{i\bar j}u_{i\bar jt}(2M-u)+(f_t+f_uu_t)(2M-u) \\[1mm]
\leq {} & \,  nu_t+(f_t+f_uu_t)(2M-u) \\[1mm]
\leq {} & \, nu_t+CM+CM|u_{t}|.
\end{split}
\]
By the Cauchy--Schwarz inequality, we obtain $u_{t}\geq-CM$ at $p_{0}$, as required.\\

The upper bound $u_{t}\leq C$ can be proved by considering the related quantity $G = \frac{u_{t}}{2M+u}$. Let $p_{0}=(z_{0},t_{0})$ be the maximum point of $G$ in $\bar{Q}_{T}$; it suffices to show $u_{t}(p_{0})\leq CM$. Repeating the arguments above, we may assume that $p_{0}$ is an interior point. Taking derivatives of $G$ at $p_0$ then similarly gives
\[
\frac{u_{t}^{2}}{(2M+u)^{2}}-\frac{u_{tt}}{2M+u} \leq 0, \ \ \
\left(\frac{u_{ti\bar j}}{2M+u}-\frac{u_tu_{i\bar j}}{(2M+u)^2}\right) \leq 0,
\]
and so
\begin{equation}\label{u t estimate eqn 2}
u_t^2 \leq u_{tt}(2M+u), \ \ \
u^{i\ov{j}}u_{ti\ov{j}}(2M+u) \leq u^{i\ov{j}}u_{i\ov{j}}u_{t} = nu_{t}.
\end{equation}
Without loss of generality, we assume that $\MA(u)>e^{2}$ at $p_{0}$ (otherwise $\MA(u)\leq e^{2}$ implies $F(u)\leq C$ and so $u_{t}=F(u)-f\leq C+K_{1}+K_{2}M$ as required). Then $F(u)=\MA(u)^{1/p}$ near $p_{0}$ and apply $\de_{t}$ to \eqref{parabolic Dirichlet problem Q T} now gives,
\[
\frac{1}{p} \cdot \MA(u)^{\frac{1}{p}} \cdot u^{i\bar j}u_{i\bar jt}-u_{tt} = f_t+f_uu_t.
\]
Using \eqref{parabolic Dirichlet problem Q T} and \eqref{u t estimate eqn 2} now gives,
\[
\begin{split}
u_t^2 \leq {} & u_{tt}(2M+u) \\[1mm]
= {} & \frac{1}{p} \cdot (u_{t}+f) \cdot u^{i\bar j}u_{i\bar jt}(2M+u)-(f_t+f_uu_t)(2M+u) \\[1mm]
\leq {} & \frac{1}{p} \cdot (u_{t}+f) \cdot nu_t-(f_t+f_uu_t)(2M+u) \\[1mm]
\leq {} & \frac{n}{p} \cdot u_{t}^{2}+CM+CM|u_{t}|.
\end{split}
\]
Recalling $p> n$ and using the Cauchy--Schwarz inequality, we obtain $u_{t}\leq CM$ at $p_{0}$, as required.
\end{proof}

\begin{corollary}\label{MA estimate}
Let $u$ be a smooth solution of \eqref{parabolic Dirichlet problem Q T}. Then there exists a positive constant $C$ depending only on $u_{0}$, $K_{1}$, $K_{2}$, $K_{3}$, $\mu$ and $(\Omega,\omega)$ such that
\[
e^{-CM} \leq \MA(u) \leq CM^{p},
\]
and
\[
C^{-1} \leq \mu'(\MA(u)) \cdot \MA(u) \leq CM,
\]
where $M=\|u\|_{C^{0}(\ov{Q}_T)}+1$.
\end{corollary}

\begin{proof}
For notational convenience, we write $\MA=\MA(u)$. By Proposition \ref{u t estimate},
\begin{equation}\label{MA estimate eqn}
|\mu(\MA)| = |u_{t}+f| \leq CM.
\end{equation}
Up to increasing $C$, the bounds on $\MA$ are immediate, e.g. either $\MA \leq e^2 \leq CM^p$ or $\mu(\MA) = \MA^{1/p}$ and $\MA \leq CM^p$. To estimate $\mu'(\MA)\cdot\MA$, there are three cases:
\begin{itemize}\setlength{\itemsep}{1mm}
\item if $\MA\in(0,1)$, then $\mu'(\MA)\cdot\MA=1$.\\[-3mm]
\item if $\MA\in[1,e^{2}]$, then
\[
\min_{s\in[1,e^{2}]}\mu'(s) \leq \mu'(\MA)\cdot\MA \leq \max_{s\in[1,e^{2}]}\mu'(s)\cdot e^{2}.
\]
\item if $\MA\in(e^{2},\infty)$, then
\[
\frac{1}{p}\cdot e^{\frac{2}{p}} \leq \frac{1}{p}\cdot\MA^{\frac{1}{p}} = \mu'(\MA)\cdot\MA \leq \frac{1}{p}\cdot |u_t + f| \leq CM.\qedhere
\]
\end{itemize}
\end{proof}

\subsection{Gradient estimate} We start by bounding the gradient along the parabolic boundary.

\begin{lemma}\label{boundary gradient estimate}
Let $u$ be a smooth solution of \eqref{parabolic Dirichlet problem Q T}. Then there exists a constant $C$ depending only on $u_{0}$, $K_{1}$, $K_{2}$, $\mu$ and $(\Omega,\omega)$ such that
\[
\sup_{\de^{*}Q_{T}}|\de u| \leq CM^{p},
\]
where $M=\|u\|_{C^{0}(\ov{Q}_T)}+1$.
\end{lemma}

\begin{proof}

Let $\rho$ be the defining function of $\Omega$ and set
\[
\underline{u}=u_{0}+({A}M)^{\frac{p}{n}}\rho.
\]
We claim that the constant ${A}$ can be chosen sufficiently large such that
\begin{equation}\label{boundary gradient estimate claim}
\underline{u} \leq u \ \ \text{in $Q_{T}$}.
\end{equation}
Given this claim, for any fixed $t\in(0,T)$, we have
\[
\underline{u} \leq u \leq 0 \ \ \mbox{in $\Omega_{t}$}, \ \ \
\underline{u} = u = 0 \ \ \mbox{on $\de\Omega_{t}$}.
\]
This implies
\[
\sup_{\de\Omega_{t}}|\de u| \leq \sup_{\de\Omega_{t}}|\de \underline{u}| \leq CM^{\frac{p}{n}} \leq CM^{p}.
\]
Since $t\in(0,T)$ is arbitrary, then $\sup_{\bar{\Gamma}_{T}}|\de u| \leq CM^{p}$ and so
\[
\sup_{\de^{*}Q_{T}}|\de u|
= \max\left(\sup_{\Omega_{0}}|\de u_{0}|,\, \sup_{\bar{\Gamma}_{T}}|\de u|\right) \leq CM^{p}.
\]

\medskip

Now we prove the claim \eqref{boundary gradient estimate claim}. It follows from $\ddbar u_{0}\geq0$ and $\ddbar\rho\geq\e_{0}\omega$ that
\[
F(\underline{u}) = \mu(\MA(\underline{u})) \geq \mu({A}^{p}M^{p}\ve_{0}^{n}).
\]
Chose ${A}$ sufficiently large such that ${A}^{p}\ve_{0}^{n} \geq e^{2}$ and ${A}\ve_{0}^{\frac{n}{p}} \geq K_{1}+K_{2}$. Then
\[
\mu({A}^{p}M^{p}\ve_{0}^{n}) = {A}M\ve_{0}^{\frac{n}{p}} \geq (K_{1}+K_{2})M  \geq \abs{f(z, t, u)}.
\]
It follows that
\[
\begin{cases}
\ F(u)-u_{t} \leq F(\underline{u})-\underline{u}_{t} & \mbox{in $Q_{T}$}, \\[1mm]
\ u \geq \underline{u} & \mbox{on $\p^{*}Q_{T}$}.
\end{cases}
\]
By the comparison principle, we obtain \eqref{boundary gradient estimate claim}.
\end{proof}

We now bound the gradient in the interior, using the maximum principle.

\begin{proposition}\label{gradient estimate}
Let $u$ be a smooth solution of \eqref{parabolic Dirichlet problem Q T}. Then there exists a constant $C$ depending only on $u_{0}$, $K_{1}$, $K_{2}$, $K_{3}$, $\mu$ and $(\Omega,\omega)$ such that
\[
\sup_{\bar{Q}_{T}}|\de u| \leq CM^{p},
\]
where $M=\|u\|_{C^{0}(\ov{Q}_T)}+1$.
\end{proposition}
\begin{proof}
Consider the quantity
\[
G = \log|\de u|^{2}+\frac{1}{2}\log(u^{2}+5M^{2})+A\rho,
\]
where $A$ is a constant to be determined later. Let $p_{0}=(z_{0},t_{0})$ be a maximum point of $G$ in $\bar{Q}_{T}$. By Lemma \ref{boundary gradient estimate}, we assume without loss of generality that $p_{0}\notin\de^{*}Q_{T}$. Then $z_{0}$ is an interior point of $\Omega_{t_{0}}$. For notational convenience, we write
\[
\mu' = \mu'\big(\MA(u)\big), \ \ \ \MA=\MA(u).
\]
Define the parabolic operator
\begin{equation}
\label{definition of L}
L = F^{i\ov{j}}\nabla_{\ov{j}}\nabla_{i}-\de_{t}  \quad  \text{where}\quad \ F^{i\ov{j}} = \mu'\cdot\MA\cdot u^{i\ov{j}}.
\end{equation}
By the maximum principle, we have $LG\leq0$ at $p_{0}$. Choose a holomorphic normal coordinate system $(U,\{z_{i}\}_{i=1}^{n})$ centered at $z_{0}$ such that
\[
g_{i\ov{j}} = \delta_{ij}, \ \ \
\de_{k}g_{i\ov{j}} = 0, \ \ \
u_{i\ov{j}} = u_{i\ov{i}}\delta_{ij} \ \ \
\text{at $z_{0}$}.
\]
The proposition will follow if we can show $|\de u|^{2}\leq CM^{2p}$ at $p_{0}$; indeed, if this holds, then we have:
\[
\log\abs{\p u}^2 = G - \frac{1}{2}\log(u^2 + 5M^2) - A\rho \leq \log \abs{\p u(p_0)}^2 + \frac{1}{2}\log \frac{6M^2}{5M^2} + C \leq \log M^{2p} + C,
\]
which implies $\abs{\p u} \leq C M^{p}$ on all of $\ov{Q}_T$.

We proceed to bond $\abs{\p u}^2$ at $p_0$. Without loss of generality, we assume that $|\de u|^{2}\geq 1$. We compute
\begin{equation}\label{gradient estimate eqn 4}
\begin{split}
0 \geq LG
& = \frac{L(|\de u|^{2})}{|\de u|^{2}}-\frac{F^{i\ov{i}}|(|\de u|^{2})_{i}|^{2}}{|\de u|^{4}}+AF^{i\ov{i}}\rho_{i\ov{i}} \\
& +\frac{uF^{i\ov{i}}u_{i\ov{i}}}{u^{2}+5M^{2}}+\left(1-\frac{2u^{2}}{u^{2}+5M^{2}}\right)\frac{F^{i\ov{i}}|u_{i}|^{2}}{u^{2}+5M^{2}}-\frac{uu_{t}}{u^{2}+5M^{2}}.
\end{split}
\end{equation}
Let us deal with the three terms in the second line. By Corollary \ref{MA estimate}, we have
\[
F^{i\ov{i}}=\mu'\cdot\MA\cdot u^{i\ov{i}}\leq CM u^{i\ov{i}},
\]
and then
\[
\frac{uF^{i\ov{i}}u_{i\ov{i}}}{u^{2}+5M^{2}} \geq -\frac{M\cdot CM \cdot n}{u^{2}+5M^{2}} \geq -C.
\]
It is clear that
\[
\left(1-\frac{2u^{2}}{u^{2}+5M^{2}}\right)\frac{F^{i\ov{i}}|u_{i}|^{2}}{u^{2}+5M^{2}}
\geq \left(1-\frac{2M^{2}}{5M^{2}}\right)\frac{F^{i\ov{i}}|u_{i}|^{2}}{u^{2}+5M^{2}}
\geq \frac{F^{i\ov{i}}|u_{i}|^{2}}{2(u^{2}+5M^{2})}.
\]
By Proposition \ref{u t estimate}, we have $|u_{t}|\leq CM$ and so
\[
-\frac{uu_{t}}{u^{2}+5M^{2}} \geq -\frac{M\cdot CM}{u^{2}+5M^{2}} \geq -C.
\]
Substituting the above into \eqref{gradient estimate eqn 4},
\begin{equation}\label{gradient estimate eqn 1}
0 \geq \frac{L(|\de u|^{2})}{|\de u|^{2}}-\frac{F^{i\ov{i}}|(|\de u|^{2})_{i}|^{2}}{|\de u|^{4}}+AF^{i\ov{i}}\rho_{i\ov{i}}+\frac{F^{i\ov{i}}|u_{i}|^{2}}{2(u^{2}+5M^{2})}-C.
\end{equation}
We now discuss the first term in the above. By direct calculation, we have
\[
\begin{split}
L(|\de u|^{2}) = {} & \sum_{k}\left(F^{i\ov{i}}u_{k}u_{\ov{k}i\ov{i}}-u_{k}u_{\ov{k}t}+F^{i\ov{i}}u_{\ov{k}}u_{ki\ov{i}}-u_{\ov{k}}u_{kt}\right)
+\sum_{k}F^{i\ov{i}}(|u_{ik}|^{2}+|u_{i\ov{i}}|^{2}). \\
\end{split}
\]
Since $(\Omega,\omega)$ is K\"ahler, we have the commutation formula for the covariant derivatives:
\[
u_{\ov{k}i\ov{i}} = u_{i\ov{i}\ov{k}}, \ \ \
u_{ki\ov{i}} = u_{ik\ov{i}} = u_{i\ov{i}k}+\sum_{p}R_{i\ov{i}k\ov{p}}u_{p}
\]
and so
\begin{align*}
&L(|\de u|^{2}) \\
= {} & \sum_{k}\left(F^{i\ov{i}}u_{k}u_{i\ov{i}\ov{k}}-u_{k}u_{\ov{k}t}+F^{i\ov{i}}u_{\ov{k}}u_{i\ov{i}k}-u_{\ov{k}}u_{kt}\right) +\sum_{k}F^{i\ov{i}}R_{i\ov{i}k\ov{p}}u_{p}u_{\ov{k}} +\sum_{k}F^{i\ov{i}}(|u_{ik}|^{2}+|u_{i\ov{i}}|^{2}) \\
= {} & 2\sum_{k}\mathrm{Re}\left(F^{i\ov{i}}u_{k}u_{i\ov{i}\ov{k}}-u_{k}u_{t\ov{k}}\right) +\sum_{k}F^{i\ov{i}}R_{i\ov{i}k\ov{p}}u_{p}u_{\ov{k}}+\sum_{k}F^{i\ov{i}}(|u_{ik}|^{2}+|u_{i\ov{i}}|^{2}) \\
\geq {} & 2\sum_{k}\mathrm{Re}\left(F^{i\ov{i}}u_{k}u_{i\ov{i}\ov{k}}-u_{k}u_{t\ov{k}}\right)-C|\de u|^{2}\sum_{i}F^{i\ov{i}}+\sum_{k}F^{i\ov{i}}|u_{ik}|^{2}.
\end{align*}
To eliminate the third order terms, apply $\nabla_{\ov{k}}$ to the flow \eqref{parabolic Dirichlet problem Q T} to see
\[
F^{i\ov{i}}u_{i\ov{i}\ov{k}}-u_{t\ov{k}} = f_{\ov{k}}+f_{u}u_{\ov{k}}.
\]
Combining the above, using Cauchy--Schwarz and the bounds on $f_u, f_k$, we obtain
\begin{equation}\label{gradient computation}
L(|\de u|^{2}) \geq \sum_{k}F^{i\ov{i}}|u_{ik}|^{2}-C|\de u|^{2}\sum_{i}F^{i\ov{i}}-C|\de u|^{2}.
\end{equation}
We now substitute this into \eqref{gradient estimate eqn 1}, using in addition that $\ddbar\rho\geq\ve_{0}\omega$, to get:
\begin{equation}
0 \geq \frac{\sum_{k}F^{i\ov{i}}|u_{ik}|^{2}}{|\de u|^{2}}
-\frac{F^{i\ov{i}}|(|\de u|^{2})_{i}|^{2}}{|\de u|^{4}}
+\frac{F^{i\ov{i}}|u_{i}|^{2}}{2u^{2}+10M^{2}}+(A\ve_{0}-C)\sum_{i}F^{i\ov{i}}-C.
\end{equation}
Choosing $A$ sufficiently large such that $A\ve_{0}-C\geq1$, and recalling that $F^{i\ov{i}}=\mu'\cdot\MA\cdot u^{i\ov{i}}$:
\begin{equation}\label{gradient estimate eqn 2}
0 \geq \frac{\sum_{k}u^{i\ov{i}}|u_{ik}|^{2}}{|\de u|^{2}}
-\frac{u^{i\ov{i}}|(|\de u|^{2})_{i}|^{2}}{|\de u|^{4}}
+\frac{u^{i\ov{i}}|u_{i}|^{2}}{2u^{2}+10M^{2}}+\sum_{i}u^{i\ov{i}}-\frac{C}{\mu'\cdot\MA}.
\end{equation}
To simplify the negative second term, we apply Cauchy--Schwarz to the first term as:
\[
\begin{split}
|\de u|^{2}\sum_{k}u^{i\ov{i}}|u_{ik}|^{2}
\geq {} & u^{i\ov{i}}\bigg|\sum_{k}u_{ik}u_{\ov{k}}\bigg|^{2} \\[-1.6mm]
= {} & u^{i\ov{i}}\bigg|(|\de u|^{2})_{i}-\sum_{k}u_{k}u_{i\ov{k}}\bigg|^{2} \\[-1mm]
= {} & u^{i\ov{i}}\left|(|\de u|^{2})_{i}-u_{i}u_{i\ov{i}}\right|^{2} \\[2mm]
= {} & u^{i\ov{i}}|(|\de u|^{2})_{i}|^{2}-2\mathrm{Re}\big((|\de u|^{2})_{i}u_{\ov{i}}\big)+|u_{i}|^{2}u_{i\ov{i}} \\[2mm]
\geq {} & u^{i\ov{i}}|(|\de u|^{2})_{i}|^{2}-2\mathrm{Re}\big((|\de u|^{2})_{i}u_{\ov{i}}\big),
\end{split}
\]
which implies
\[
\frac{\sum_{k}u^{i\ov{i}}|u_{ik}|^{2}}{|\de u|^{2}}
\geq \frac{u^{i\ov{i}}|(|\de u|^{2})_{i}|^{2}}{|\de u|^{4}}-\frac{2\mathrm{Re}\big((|\de u|^{2})_{i}u_{\ov{i}}\big)}{|\de u|^{4}}.
\]
Substituting this into \eqref{gradient estimate eqn 2} and simplifying gives,
\begin{equation}\label{gradient estimate eqn new}
0 \geq -\frac{2\mathrm{Re}\big((|\de u|^{2})_{i}u_{\ov{i}}\big)}{|\de u|^{4}}
+\frac{u^{i\ov{i}}|u_{i}|^{2}}{2u^{2}+10M^{2}}+\sum_{i}u^{i\ov{i}}-\frac{C}{\mu'\cdot\MA}.
\end{equation}
Since $G_{i}=0$ at $p_{0}$, we have
\[
{\frac{(|\de u|^{2})_{i}}{|\de u|^{2}}+\frac{uu_{i}}{u^2+5M^2}+A\rho_{i}}=0.
\]
Using this in \eqref{gradient estimate eqn new}, along with Cauchy--Schwarz and $\mu'\cdot\MA\geq C^{-1}$ (see Corollary \ref{MA estimate}), we obtain:
\begin{align*}
0 &\geq \frac{2u|\de u|^{2}}{(u^{2}+5M^{2})|\de u|^{2}}+\frac{2\mathrm{Re}(A\rho_{i}u_{\ov{i}}\big)}{|\de u|^{2}}
+\frac{u^{i\ov{i}}|u_{i}|^{2}}{2u^{2}+10M^{2}}+\sum_{i}u^{i\ov{i}}-\frac{C}{\mu'\cdot\MA}\\
&\geq  \frac{u^{i\ov{i}}|u_{i}|^{2}}{2u^{2}+10M^{2}}+\sum_{i}u^{i\ov{i}} - C.
\end{align*}
This implies both
\begin{equation}\label{gradient estimate eqn 3}
\sum_{i}u^{i\ov{i}} \leq C \quad \text{ and }\quad u^{i\ov{i}}|u_{i}|^{2} \leq CM^{2}.
\end{equation}
Now by Corollary \ref{MA estimate},
\[
\prod_{j}u_{j\ov{j}} = \MA \leq CM^{p}.
\]
Combining this with \eqref{gradient estimate eqn 3}, we have for each $i$ that:
\[
u_{i\ov{i}} = \left(\prod_{j}u_{j\ov{j}}\right) \cdot \left(\prod_{j\neq i}u^{j\ov{j}} \right)
\leq \left(\prod_{j}u_{j\ov{j}}\right)\cdot\left(\sum_{j}u^{j\ov{j}} \right)^{n-1}
\leq CM^{p},
\]
so that
\[
u^{i\ov{i}} \geq \frac{1}{CM^{p}}.
\]
Using \eqref{gradient estimate eqn 3} again,
\[
\frac{|\de u|^{2}}{CM^{p}} \leq u^{i\ov{i}}|u_{i}|^{2} \leq CM^{2}.
\]
Combining this with $p > n + 1 \geq 2$ gives
\[
|\de u|^{2} \leq CM^{p + 2} \leq CM^{2p},
\]
as required.
\end{proof}

\subsection{Hessian estimate} We are left to bound the (real) spatial Hessian of $u$. We again start with the boundary estimate:
\begin{lemma}\label{boundary Hessian estimate}
Let $u$ be a smooth solution of \eqref{parabolic Dirichlet problem Q T}. Then there exists a constant $C_{M}$, depending only on $u_{0}$, $f$, $\mu$, $(\Omega,\omega)$, and an upper bound for $\|u\|_{C^{0}(\ov{Q}_T)}$, such that
\[
\sup_{\de^{*}Q_{T}}|\nabla^{2}u| \leq C_{M}.
\]
\end{lemma}

\begin{proof}
For notational convenience, we denote the constant $C_{M}$ by just $C$ throughout the proof. Since $u=u_{0}$ in $\Omega_{0}$, then
\[
\sup_{\Omega_{0}}|\nabla^{2}u| = \sup_{\Omega_{0}}|\nabla^{2}u_{0}| \leq C.
\]
Since $\de^{*}Q_{T}=\Omega_{0}\cup\bar{\Gamma}_{T}$, it will suffice to bound the Hessian on $\ov{\Gamma}_T$. Fix a point $(z_{0},t_{0})\in\bar{\Gamma}_{T}=\de\Omega\times[0,T]$ then, and let $(\{z_{i}\}_{i=1}^{n},B_{R})$ be a holomorphic coordinate system centered at $z_{0}$ such that
\begin{itemize}\setlength{\itemsep}{1mm}
\item $B_{R}$ is an Euclidean ball of radius $R$;
\item $g_{i\ov{j}}(0)=\delta_{ij}$ and $\frac{1}{2}(\delta_{ij})\leq (g_{i\ov{j}})\leq 2(\delta_{ij})$ in $B_{R}\cap\Omega$;
\item $\rho_{x_{\alpha}}(0)=0$ for $\alpha=1,\ldots,2n-1$, where $z_i = x_{2i-1}+\sqrt{-1}x_{2i}$ for $i=1,\ldots,n$; and
\item $x_{2n} > 0$ in $B_{R}\cap\Omega$.
\end{itemize}
After rescaling $\rho$ by a positive constant if necessary, we can assume
\[
\rho(z) = -x_{2n} + O(|z|^2) \ \ \text{in $B_{R}\cap\Omega$}.
\]
For $\alpha=1,\ldots,2n-1$, define the tangential operators:
\[
T_\alpha = \frac{\p}{\p x_\alpha} - \frac{\rho_{x_\alpha}}{\rho_{x_{2n}}}\frac{\p}{\p x_{2n}}.
\]
We split the argument into three steps.\\

\noindent
{\bf Step 1.} Bounding the tangent-tangent derivatives.\\

Since $u\equiv0$ on $\de\Omega_{t_{0}}$, we have $T_{\alpha}T_{\beta}u(0)=0$ for $\alpha,\beta=1,\ldots,2n-1$. This implies
\[
\p_{x_\alpha}\p_{x_\beta}u(0)+\p_{x_{2n}}u(0)\cdot\p_{x_\alpha}\p_{x_\beta} \rho(0) = 0.
\]
Combining this with Proposition \ref{gradient estimate} gives
\[
|\p_{x_\alpha}\p_{x_\beta}u(0)|  \leq C.
\]

\noindent
{\bf Step 2.} Bounding the tangent-normal derivatives.\\

Set $U=u-u_{0}$ and consider the following functions:
\[
w_{\pm} = -A\rho+K|z|^{2}\pm T_{\alpha}U-(\de_{x_{2n-1}}U)^{2},
\]
where $A$ and $K$ are constants to be determined later. We claim that
\begin{equation}\label{tangent normal claim}
\begin{cases}
\ Lw_{\pm} \leq 0 & \mbox{in $(B_{R}\cap\Omega)\times(0,T)$}, \\[1mm]
\ w_{\pm} \geq 0 & \mbox{on $\de^{*}\big((B_{R}\cap\Omega)\times(0,T)\big)$},
\end{cases}
\end{equation}
where $L$ is the operator defined by \eqref{definition of L}. Given this claim, the comparison principle shows that $w_{\pm}\geq0$ in $(B_{R}\cap\Omega)\times(0,T)$ and so
\[
\pm T_{\alpha}U \leq -A\rho+K|z|^{2}-(\de_{x_{2n-1}}U)^{2} \leq -A\rho+K|z|^{2}.
\]
At $0$, observe that both sides are $0$ and so
\[
|\de_{x_{2n}}T_{\alpha}U(0)| \leq |\de_{x_{2n}}(-A\rho+K|z|^{2})(0)| \leq C.
\]
By the definition of $T_{\alpha}$ and Proposition \ref{gradient estimate},
\begin{align*}
|\de_{x_{2n}}T_{\alpha}U(0)|
= {} & |\de_{x_{2n}}T_{\alpha}u(0)-\de_{x_{2n}}T_{\alpha}u_{0}(0)| \\
= {} & |\de_{x_{2n}}\de_{x_{\alpha}}u(0)+\de_{x_{2n}}\de_{x_{\alpha}}\rho(0)\cdot\de_{x_{2n}}u(0)-\de_{x_{2n}}T_{\alpha}u_{0}(0)| \\
\geq {} & |\de_{x_{2n}}\de_{x_{\alpha}}u(0)|-C.
\end{align*}
Then we obtain
\[
|\de_{x_{2n}}\de_{x_{\alpha}}u(0)| \leq |\de_{x_{2n}}T_{\alpha}U(0)|+C \leq C,
\]
as required. We are left to verify the two parts of the claim \eqref{tangent normal claim}; we do so only for $w_{+}$, as the argument for $w_{-}$ is almost identical.

\bigskip
\noindent
{\bf Step 2.1.} $w_{+}\geq0$ on $\de^{*}\big((B_{R}\cap\Omega)\times(0,T)\big)$.
\bigskip

Notice that
\begin{align*}
& \de^{*}\big((B_{R}\cap\Omega)\times(0,T)\big) \\
= {} & \big((B_{R}\cap\Omega)\times\{0\}\big) \cup \big((B_{R}\cap\de\Omega)\times[0,T]\big) \cup
\big((\de B_{R}\cap\Omega)\times[0,T]\big).
\end{align*}
On $\big((B_{R}\cap\Omega)\times\{0\}\big)$, we have $U\equiv0$, which implies $T_{\alpha}U\equiv0$ and $\de_{x_{2n}}U\equiv0$, and then
\[
w_{+} = -A\rho+C|z|^{2} \geq 0.
\]
Since $U\equiv0$ on $\de\Omega\times[0,T]$, then $T_{2n-1}U=0$ and so
\[
|\de_{x_{2n-1}}U| = \left|\frac{\rho_{x_{2n-1}}}{\rho_{x_{2n}}}\cdot\de_{x_{2n}}U\right|
\leq C|\rho_{x_{2n-1}}| \leq C|z| \ \ \text{on $(B_{R}\cap\de\Omega)\times[0,T]$},
\]
where we used Proposition \ref{gradient estimate} and $\rho_{x_{2n-1}}(0)=0$. On $(B_{R}\cap\de\Omega)\times[0,T]$, using $T_{\alpha}U\equiv0$ and choosing $K$ sufficiently large, we have
\[
w_{+} = -A\rho+K|z|^{2}-(\de_{x_{2n-1}}U)^{2}
\geq (K-C)|z|^{2} \geq 0.
\]
On $(\de B_{R}\cap\Omega)\times[0,T]$, by Proposition \ref{gradient estimate} and increasing $K$ if needed, we have
\[
w_{+} \geq -A\rho+KR^{2}-C \geq 0.
\]
Notice that the value of $K$ is fixed now.

\bigskip
\noindent
{\bf Step 2.2.} $Lw_{+}\leq 0$ in $(B_{R}\cap\Omega)\times(0,T)$.
\bigskip

Since $(\Omega,\omega)$ is K\"ahler, then
\[
L = F^{i\ov{j}}\nabla_{\ov{j}}\nabla_{i}-\de_{t} = F^{i\ov{j}}\de_{\ov{j}}\de_{i}-\de_{t}.
\]
By direct calculation,
\begin{equation}\label{L w eqn 1}
\begin{split}
Lw_{+} = {} & -AF^{i\ov{j}}\rho_{i\ov{j}}+K\sum_{i}F^{i\ov{i}}+L(T_{\alpha}U) \\
& -2(\de_{x_{2n-1}}U)L\big(\de_{x_{2n-1}}U\big)-2F^{i\ov{j}}(\de_{i}\de_{x_{2n-1}}U)(\de_{\ov{j}}\de_{x_{2n-1}}U).
\end{split}
\end{equation}
For the third term in \eqref{L w eqn 1}, it is clear that
\begin{equation}\label{L w eqn 2}
L(T_{\alpha}U) = L(T_{\alpha}u)-L(T_{\alpha}u_{0}) \leq L(T_{\alpha}u)+C\sum_{i}F^{i\ov{i}}.
\end{equation}
For convenience, write $a=-\frac{\rho_{x_\alpha}}{\rho_{x_{2n}}}$, so that $T_{\alpha}=\de_{x_{\alpha}}+a\de_{x_{2n}}$. Then
\begin{equation}\label{L w eqn 3}
\begin{split}
L(T_{\alpha}u & ) =  F^{i\ov{j}}\de_{\ov{j}}\de_{i}(\de_{x_{\alpha}}u+a\de_{x_{2n}}u)
-\de_{t}(\de_{x_{\alpha}}u+a\de_{x_{2n}}u)  \\
= {} & F^{i\ov{j}}T_{\alpha}(u_{i\ov{j}})+2\mathrm{Re}\big(F^{i\ov{j}}a_{i}\de_{\ov{j}}\de_{x_{2n}}u\big)
+(F^{i\ov{j}}a_{i\ov{j}})\de_{x_{2n}}u-T_{\alpha}(u_{t}) \\
\leq {} & F^{i\ov{j}}T_{\alpha}(u_{i\ov{j}})-T_{\alpha}(u_{t})
+2\mathrm{Re}\big(F^{i\ov{j}}a_{i}\de_{\ov{j}}\de_{x_{2n}}u\big)+C\sum_{i}F^{i\ov{i}}.
\end{split}
\end{equation}
Applying $T_{\alpha}$ to the equation in \eqref{parabolic Dirichlet problem Q T},
\[
\mu'(\MA(u))\cdot T_{\alpha}\left(\frac{\det u_{i\ov{j}}}{\det g_{i\ov{j}}}\right)-T_{\alpha}(u_{t}) = T_{\alpha}\big(f(z,t,u)\big).
\]
Combining this with Corollary \ref{MA estimate}, we obtain
\begin{equation}\label{L w eqn 4}
\begin{split}
& F^{i\ov{j}}T_{\alpha}(u_{i\ov{j}})-T_{\alpha}(u_{t}) \\
= {} & T_{\alpha}\big(f(z,t,u)\big)-\mu'(\MA(u))\cdot\det u_{i\ov{j}}\cdot T_{\alpha}(\det g^{i\ov{j}}) \\
= {} & T_{\alpha}\big(f(z,t,u)\big)-\mu'(\MA(u))\cdot\MA(u)\cdot \det g_{i\ov{j}}\cdot T_{\alpha}(\det g^{i\ov{j}}) \\[1mm]
\leq {} & C.
\end{split}
\end{equation}
Since $\de_{x_{2n}}=\sqrt{-1}(2\de_{z_{n}}-\de_{x_{2n-1}})$, then
\[
\de_{\ov{j}}\de_{x_{2n}}u = \sqrt{-1}(2u_{n\ov{j}}-\de_{\ov{j}}\de_{x_{2n-1}}u)
\]
and so
\[
\begin{split}
2\mathrm{Re}\big(u^{i\ov{j}}a_{i}\de_{\ov{j}}\de_{x_{2n}}u\big)
= {} & 2\mathrm{Im}\big(u^{i\ov{j}}a_{i}(\de_{\ov{j}}\de_{x_{2n-1}}u-2u_{n\ov{j}})\big) \\[0.6mm]
= {} & 2\mathrm{Im}\big(u^{i\ov{j}}a_{i}\de_{\ov{j}}\de_{x_{2n-1}}u\big)-4\mathrm{Im}(u^{i\ov{j}}a_{i}u_{n\ov{j}}) \\[0.6mm]
= {} & 2\mathrm{Im}\big(u^{i\ov{j}}a_{i}\de_{\ov{j}}\de_{x_{2n-1}}(U-u_{0})\big)-4\mathrm{Im}a_{n} \\[0.6mm]
\leq {} & 2\big|u^{i\ov{j}}a_{i}\de_{\ov{j}}\de_{x_{2n-1}}U\big|+C\sum_{i}u^{i\ov{i}}+C.
\end{split}
\]
Combining this with $F^{i\ov{j}}=\mu'\cdot\MA\cdot u^{i\ov{j}}$ and Corollary \ref{MA estimate},
\begin{equation}\label{L w eqn 5}
\begin{split}
2\mathrm{Re}\big(F^{i\ov{j}}a_{i}\de_{\ov{j}}\de_{x_{2n}}u\big)
\leq {} & 2\big|F^{i\ov{j}}a_{i}\de_{\ov{j}}\de_{x_{2n-1}}U\big|+C\sum_{i}F^{i\ov{i}}+C\cdot\mu'\cdot\MA \\
\leq {} & 2\big|F^{i\ov{j}}a_{i}\de_{\ov{j}}\de_{x_{2n-1}}U\big|+C\sum_{i}F^{i\ov{i}}+C.
\end{split}
\end{equation}
Substituting \eqref{L w eqn 3}, \eqref{L w eqn 4} and \eqref{L w eqn 5} into \eqref{L w eqn 2},
\begin{equation}\label{L w eqn 6}
L(T_{\alpha}U) \leq 2|F^{i\ov{j}}a_{i}\de_{\ov{j}}\de_{x_{2n-1}}U|+C\sum_{i}F^{i\ov{i}}+C.
\end{equation}
For the term $-2(\de_{x_{2n-1}}U)L\big(\de_{x_{2n-1}}U\big)$ in \eqref{L w eqn 1}, applying $\de_{x_{2n-1}}$ to \eqref{parabolic Dirichlet problem Q T} and arguing similarly to how we showed \eqref{L w eqn 4} above, we obtain
\[
\left|F^{i\ov{j}}\de_{x_{2n-1}}(u_{i\ov{j}})-\de_{x_{2n-1}}(u_{t})\right| \leq C.
\]
Thus
\[
\begin{split}
\left|L\big(\de_{x_{2n-1}}U\big)\right|
= {} & \left|L\big(\de_{x_{2n-1}}u\big)-L\big(\de_{x_{2n-1}}u_{0}\big)\right| \\[1mm]
= {} & \left|F^{i\ov{j}}\de_{x_{2n-1}}(u_{i\ov{j}})-\de_{x_{2n-1}}(u_{t})
-F^{i\ov{j}}\de_{i}\de_{\ov{j}}(\de_{x_{2n-1}}u_{0})\right|  \\[1mm]
\leq {} & C+C\sum_{i}F^{i\ov{i}}.
\end{split}
\]
Then
\begin{equation}\label{L w eqn 7}
-2(\de_{x_{2n-1}}U)L\big(\de_{x_{2n-1}}U\big) \leq -C\sum_{i}F^{i\ov{i}}-C.
\end{equation}
Substituting \eqref{L w eqn 6} and \eqref{L w eqn 7} into \eqref{L w eqn 1},
\[
\begin{split}
Lw_{+} \leq {} & -AF^{i\ov{j}}\rho_{i\ov{j}}+2\left|F^{i\ov{j}}a_{i}\de_{\ov{j}}\de_{x_{2n-1}}U\right|+C\sum_{i}F^{i\ov{i}}+C \\
& -2F^{i\ov{j}}(\de_{i}\de_{x_{2n-1}}U)(\de_{\ov{j}}\de_{x_{2n-1}}U).
\end{split}
\]
By the Cauchy--Schwarz inequality,
\[
\begin{split}
2\left|F^{i\ov{j}}a_{i}\de_{\ov{j}}\de_{x_{2n-1}}U\right|
\leq {} & F^{i\ov{j}}(\de_{i}\de_{x_{2n-1}}U)(\de_{\ov{j}}\de_{x_{2n-1}}U)+F^{i\ov{j}}a_{i}a_{\ov{j}} \\
\leq {} & F^{i\ov{j}}(\de_{i}\de_{x_{2n-1}}U)(\de_{\ov{j}}\de_{x_{2n-1}}U)+C\sum_{i}F^{i\ov{i}}.
\end{split}
\]
Recalling $\ddbar\rho\geq\ve_{0}\omega$ and so $\rho_{i\ov{j}}\geq\ve_{0}g_{i\ov{j}}\geq(\ve_{0}/2)\delta_{ij}$,
\[
Lw_{+} \leq -\left(\frac{A\ve_{0}}{2}-C\right)\sum_{i}F^{i\ov{i}}+C.
\]
From Corollary \ref{MA estimate} and the arithmetic-geometric mean inequality, we obtain
\[
\sum_{i}F^{i\ov{i}} = \mu'\cdot\MA\cdot\sum_{i}u^{i\ov{i}} \geq \mu'\cdot\MA\cdot n\left(\det(u^{i\ov{j}})\right)^{\frac{1}{n}} \geq C^{-1}.
\]
After choosing $A$ sufficiently large, we are done.\\

\noindent
{\bf Step 3.} Bounding the normal-normal derivative.
\bigskip

Using Corollary \ref{MA estimate} and Steps 1 and 2, we see that
\[
\left|
\det\left([u_{i\ov{j}}(0)]_{1\leq i,j\leq n}\right)
-u_{n\ov{n}}(0)\cdot\det\left([u_{i\ov{j}}(0)]_{1\leq i,j\leq n-1}\right)
\right| \leq C
\]
and so
\[
u_{n\ov{n}}(0)\cdot\det\left([u_{i\ov{j}}(0)]_{1\leq i,j\leq n-1}\right) \leq C.
\]
We claim that
\begin{equation}\label{normal normal claim}
\det\left([u_{i\ov{j}}(0)]_{1\leq i,j\leq n-1}\right) \geq C^{-1}.
\end{equation}
This claim implies $u_{n\ov{n}}(0)\leq C$. Combining it with Steps 1 and 2, we obtain $|u_{x_{2n}x_{2n}}(0)|\leq C$, as required.

To show \eqref{normal normal claim}, we set $\ov{u}=\rho/A$. Using Corollary \ref{MA estimate} and choosing $A$ sufficiently large, we see that
\[
\begin{cases}
\ \MA(u) \geq \MA(\ov{u}) & \mbox{in $\Omega_{t_{0}}$}, \\[1mm]
\ u = \ov{u} = 0 & \mbox{on $\de\Omega_{t_{0}}$}.
\end{cases}
\]
By the comparison principle:
\[
u \leq \ov{u} \leq 0 \ \ \mbox{in $\Omega_{t_{0}}$}
\]
and so
\[
-u_{x_{2n}}(0) \geq -A\rho_{x_{2n}}(0) \geq C^{-1}.
\]
Since $u\equiv0$ on $\de\Omega_{t_{0}}$, then
\[
u_{i\ov{j}}(0) = -u_{x_{n}}(0)\cdot\rho_{i\ov{j}}(0)
\geq C^{-1}\delta_{i\ov{j}} \ \ \text{for $1\leq i,j\leq n-1$},
\]
which implies \eqref{normal normal claim}.
\end{proof}

We now bound the complex Hessian in the interior:

\begin{proposition}\label{complex Hessian estimate}
Let $u$ be a smooth solution of \eqref{parabolic Dirichlet problem Q T}. Then there exists a positive constant $C_{M}$ depending only on an upper bound of $\|u\|_{C^{0}(\ov{Q}_T)}$, $u_{0}$, $f$, $\mu$ and $(\Omega,\omega)$ such that
\[
C_{M}^{-1}g_{i\ov{j}} \leq u_{i\ov{j}} \leq C_{M}g_{i\ov{j}}
\ \ \text{in $\bar{Q}_{T}$}.
\]
\end{proposition}

\begin{proof}
For notational convenience, we denote the constant $C_{M}$ by $C$. By $\ddbar u>0$ and Corollary \ref{MA estimate}, it will suffice to show $\ddbar u\leq C\omega$ in $\bar{Q}_{T}$. For $(z,t)\in\bar{Q}_{T}$ and a unitary vector $\xi\in T_{z}^{1,0}\Omega$, we consider the quantity
\[
G(z,t,\xi) = \log u_{\xi\ov{\xi}}+A\rho,
\]
where $A$ is a constant to be determined later. Let $(z_{0},t_{0},\xi_{0})$ be the maximum point of $G$. By Lemma \ref{boundary Hessian estimate}, we assume without loss of generality that $p_{0}=(z_{0},t_{0})\notin\de^{*}Q_{T}$. Then $z_{0}$ is an interior point of $\Omega_{t_{0}}$. Choose a holomorphic normal coordinate system $(U,\{z_{i}\}_{i=1}^{n})$ centered at $z_{0}$ such that
\[
g_{i\ov{j}} = \delta_{ij}, \ \ \
\de_{k}g_{i\ov{j}} = 0, \ \ \
u_{i\ov{j}} = u_{i\ov{i}}\delta_{ij}, \ \ \
\de_{1} = \xi_{0} \ \ \
\text{at $z_{0}$}.
\]
In $U$, the quantity
\[
\hat{G}(z,t) = \log\frac{u_{1\ov{1}}}{g_{1\ov{1}}}+A\rho
\]
is well-defined near $p_{0}$ and achieves its maximum at $p_{0}$. By the maximum principle, we have $L\hat{G}\leq0$ at $p_{0}$, where $L$ is the operator defined by \eqref{definition of L}.

To prove the proposition, it will then suffice to show $u_{1\ov{1}}\leq C$ at $p_{0}$. Without loss of generality, we assume that $u_{1\ov{1}}\geq1$. We compute
\begin{equation}\label{complex Hessian estimate eqn 1}
0 \geq L\hat{G}
= \frac{L(u_{1\ov{1}})}{u_{1\ov{1}}}-\frac{F^{i\ov{i}}|u_{1\ov{1}i}|^{2}}{u_{1\ov{1}}^{2}}+Au^{i\ov{i}}\rho_{i\ov{i}}.
\end{equation}
Applying $\nabla_{\ov{1}}\nabla_{1}$ on both sides of the equation gives
\[
\nabla_{1}\left(\mu'\cdot\MA\cdot u^{i\ov{j}}u_{i\ov{j}1}\right)-u_{t1\ov{1}} = \nabla_{1}\nabla_{\ov{1}}\big(f(z,t,u)\big)
\]
and so
\begin{equation}\label{complex Hessian estimate eqn 2}
\nabla_{1}\left(\mu'\cdot\MA\cdot u^{i\ov{j}}\right)u_{i\ov{j}1}+F^{i\ov{i}}u_{i\ov{i}1\ov{1}}-u_{t1\ov{1}} = \nabla_{1}\nabla_{\ov{1}}\big(f(z,t,u)\big).
\end{equation}
Since $(\Omega,\omega)$ is K\"ahler, then the commutation formula shows
\[
\left|u_{i\ov{i}1\ov{1}}-u_{1\ov{1}i\ov{i}}\right| = \left|\sum_{p}\left(R_{1\ov{1}i\ov{p}}u_{p\ov{i}}-R_{i\ov{i}1\ov{p}}u_{p\ov{1}}\right)\right|
= \left|R_{1\ov{1}i\ov{i}}(u_{i\ov{i}}-u_{1\ov{1}})\right| \leq Cu_{1\ov{1}}.
\]
Combining this with \eqref{complex Hessian estimate eqn 2} and recalling $u_{1\ov{1}}\geq1$, we see that
\[
L(u_{1\ov{1}}) \geq \mu'\cdot\MA \cdot u^{i\ov{i}}u^{j\ov{j}}|u_{i\ov{j}1}|^{2}-(\mu''\cdot\MA^{2}+\mu'\cdot\MA)\cdot|u^{i\ov{i}}u_{i\ov{i}1}|^{2}-Cu_{1\ov{1}}.
\]
Since $\mu$ satisfies $t^{2}\mu''(t)+t\mu'(t) \leq \frac{t}{p}\mu'(t)$, then by $p>n$ and the Cauchy-Schwarz inequality, we see that
\[
\begin{split}
(\mu''\cdot\MA^{2}+\mu'\cdot\MA) \cdot |u^{i\ov{i}}u_{i\ov{i}1}|^{2}
\leq \mu'\cdot \MA \cdot \frac{1}{n}\cdot|u^{i\ov{i}}u_{i\ov{i}1}|^{2}
\leq \mu'\cdot \MA \cdot u^{i\ov{i}}u^{i\ov{i}}|u_{i\ov{i}1}|^{2}.
\end{split}
\]
This shows that
\begin{equation}\label{complex Hessian estimate eqn 3}
\begin{split}
L(u_{1\ov{1}}) \geq {} & \mu'\cdot\MA\cdot \sum_{i\neq j}u^{i\ov{i}}u^{j\ov{j}}|u_{i\ov{j}1}|^{2}-Cu_{1\ov{1}} \\
= {} & \sum_{i\neq j}F^{i\ov{i}}u^{j\ov{j}}|u_{i\ov{j}1}|^{2}-Cu_{1\ov{1}}.
\end{split}
\end{equation}
Substituting this into \eqref{complex Hessian estimate eqn 1}, and using $\ddbar\rho\geq\ve_{0}\omega$, we obtain
\[
\begin{split}
0 \geq {} & \frac{\sum_{i\neq j}F^{i\ov{i}}u^{j\ov{j}}|u_{i\ov{j}1}|^{2}}{u_{1\ov{1}}}
-\frac{F^{i\ov{i}}|u_{1\ov{1}i}|^{2}}{u_{1\ov{1}}^{2}}
+AF^{i\ov{i}}\rho_{i\ov{i}}-C\sum_{i}F^{i\ov{i}}-C \\
\geq {} & \frac{\sum_{i>1}F^{i\ov{i}}u^{1\ov{1}}|u_{i\ov{1}1}|^{2}}{u_{1\ov{1}}}
-\frac{F^{i\ov{i}}|u_{1\ov{1}i}|^{2}}{u_{1\ov{1}}^{2}}
+(A\ve_{0}-C)\sum_{i}F^{i\ov{i}}-C \\
= {} & -\frac{F^{1\ov{1}}|u_{1\ov{1}1}|^{2}}{u_{1\ov{1}}^{2}}
+(A\ve_{0}-C)\sum_{i}F^{i\ov{i}}-C.
\end{split}
\]
By $\hat{G}_{1}=0$, we see that
\[
\frac{|u_{1\ov{1}1}|^{2}}{u_{1\ov{1}}^{2}} = A^{2}|\rho_{1}|^{2} \leq CA^{2}
\]
and so
\[
(A\ve_{0}-C)\sum_{i}F^{i\ov{i}} \leq CA^{2}F^{1\ov{1}}+C.
\]
Choosing $A$ sufficiently large such that $A\ve_{0}-C\geq1$, and recalling $F^{i\ov{i}}=\mu'\cdot\MA \cdot u^{i\ov{i}}$,
\[
\sum_{i}u^{i\ov{i}} \leq Cu^{1\ov{1}}+\frac{C}{\mu'\cdot\MA}.
\]
By Corollary \ref{MA estimate} and $u_{1\ov{1}}\geq1$, we obtain $\sum_{i}u^{i\ov{i}}\leq C$ and so $u_{1\ov{1}}\leq C$, as required.
\end{proof}

We can now use an argument entirely similar to Proposition \ref{complex Hessian estimate} to conclude a bound on the full Hessian:

\begin{proposition}\label{Hessian estimate}
Let $u$ be a smooth solution of \eqref{parabolic Dirichlet problem Q T}. Then there exists a constant $C_{M}$ depending only on an upper bound of $\|u\|_{C^{0}(\ov{Q}_T)}$, $u_{0}$, $f$, $\mu$ and $(\Omega,\omega)$ such that
\[
\sup_{\bar{Q}_{T}}|\nabla^{2} u| \leq C_{M}.
\]
\end{proposition}

\begin{proof}
For notational convenience, we again denote $C_{M}$ by $C$. Since $\Delta u>0$, it suffices to establish an upper bound of $\nabla^{2}u$. For $(z,t)\in\bar{Q}_{T}$ and a unit vector $v\in T_{z}\Omega$, we consider the quantity
\[
G(z,t,v) = \nabla^{2}u(v,v)+|\de u|^{2}.
\]
Let $(z_{0},t_{0},v_{0})$ be the maximum point of $G$. By Lemma \ref{boundary Hessian estimate}, we can assume without loss of generality that $p_{0}=(z_{0},t_{0})\notin\de^{*}Q_{T}$, so that $z_{0}$ is an interior point of $\Omega_{t_{0}}$. Choose a holomorphic normal coordinate system $(U,\{z_{i}\}_{i=1}^{n})$ centered at $z_{0}$ such that
\[
g_{i\ov{j}} = \delta_{ij}, \ \ \
\de_{k}g_{i\ov{j}} = 0, \ \ \
u_{i\ov{j}} = u_{i\ov{i}}\delta_{ij} \ \ \
\text{at $z_{0}$}.
\]
In $U$, extend the vector $v_{0}$ to be a vector field $V$ by taking the components to be constant. Then the quantity
\[
\hat{G}(z,t) = \frac{\nabla^{2}u(V,V)}{|V|^{2}}+|\de u|^{2}
\]
is well-defined near $p_{0}$ and achieves its maximum at $p_{0}$. By the maximum principle, we have $L\hat{G}\leq0$ at $p_{0}$, where $L$ is the operator defined by \eqref{definition of L}. To prove Proposition \ref{Hessian estimate}, it suffices to show $|\nabla^{2}u|\leq C$ at $p_{0}$. We compute
\begin{equation}\label{Hessian estimate eqn 1}
0 \geq L\hat{G} = L(u_{VV})+L(|\de u|^{2}).
\end{equation}
For the first term of \eqref{Hessian estimate eqn 1}, applying $\nabla_{V}\nabla_{V}$ to the equation in \eqref{parabolic Dirichlet problem Q T} and using the similar argument of \eqref{complex Hessian estimate eqn 3},
\begin{equation}\label{Hessian estimate eqn 2}
\begin{split}
L(u_{VV}) \geq {} & \sum_{i\neq j}F^{i\ov{i}}u^{j\ov{j}}|u_{i\ov{j}V}|^{2}-C|\nabla^{2}u|\sum_{i}F^{i\ov{i}}-C|\nabla^{2}u| \\
\geq {} & -C|\nabla^{2}u|\sum_{i}F^{i\ov{i}}-C|\nabla^{2}u|.
\end{split}
\end{equation}
For the second term of \eqref{Hessian estimate eqn 1}, by the same computation as in Proposition \ref{gradient estimate} (see \eqref{gradient computation}),
\begin{equation}\label{Hessian estimate eqn 3}
L(|\de u|^{2}) \geq \sum_{k}F^{i\ov{i}}|u_{ik}|^{2}-C\sum_{i}F^{i\ov{i}}-C.
\end{equation}
Substituting \eqref{Hessian estimate eqn 2} and \eqref{Hessian estimate eqn 3} into \eqref{Hessian estimate eqn 1},
\[
0 \geq \sum_{k}F^{i\ov{i}}|u_{ik}|^{2}-C|\nabla^{2}u|\sum_{i}F^{i\ov{i}}-C|\nabla^{2}u|-C\sum_{i}F^{i\ov{i}}-C.
\]
From Corollary \ref{MA estimate} and Proposition \ref{complex Hessian estimate}, we have $C^{-1}\leq F^{i\ov{i}}\leq C$ and so
\[
0 \geq \sum_{k}|u_{ik}|^{2}-C|\nabla^{2}u|-C.
\]
Combining this with
\[
|\nabla^{2}u|^{2} = \sum_{i,k}(|u_{ik}|^{2}+|u_{i\ov{k}}|^{2})
\leq \sum_{i,k}|u_{ik}|^{2}+C,
\]
we obtain
\[
0 \geq |\nabla^{2}u|^{2}-C|\nabla^{2}u|-C.
\]
By the Cauchy-Schwarz inequality, we obtain $|\nabla^{2}u|\leq C$, as required.
\end{proof}

\subsection{Higher order estimates and existence}

We conclude by using standard regularity theory to establish existence of solutions to the $\mu$-flow:

\begin{proof}[Proof of Theorem \ref{thm:estimate of flow}]
By standard parabolic theory, it suffices to establish a priori estimates on the solution. Propositions \ref{uniform estimate}, \ref{u t estimate}, \ref{gradient estimate}, \ref{complex Hessian estimate} and \ref{Hessian estimate} imply the uniform $C^{2,1}$ estimate on the solution, and also that \eqref{parabolic Dirichlet problem Q T} is uniformly parabolic. It follows from Lemma \ref{concavity of F} that \eqref{parabolic Dirichlet problem Q T} is concave. Thus, the higher order estimates can be deduced using Evans--Krylov theory and a standard bootstrapping argument.
\end{proof}

\subsection{Stability}

In this section, we suppose that $u,v$ are smooth solutions to \eqref{parabolic Dirichlet problem Q T}, with initial conditions $u_0$ and $v_0$, respectively. Following a standard parabolic argument (for example, see \cite[Proposition 1.5]{To17}), we show:

\begin{lemma}\label{stability lemma}
Set $M_{0} := \max\left\{\|u\|_{C^{0}(Q_{T})},\|v\|_{C^{0}(Q_{T})}\right\}$ and
\[
K := \max_{s\in[-M_{0},M_{0}]}\left|\frac{\de f}{\de s}(z,t,s)\right|.
\]
Then we have
\[
\|u-v\|_{C^{0}(Q_{T})} \leq e^{KT}\|u_{0}-v_{0}\|_{C^{0}(\Omega)}.
\]
\end{lemma}

\begin{proof}
Define
\[
w = e^{-Kt}(u-v)
\]
and let $(z_{0},t_{0})$ be the maximum point of $w$. It suffices to show
\begin{equation}\label{stability eqn 1}
w(z_{0},t_{0}) \leq \|u_{0}-v_{0}\|_{C^{0}(\Omega)}.
\end{equation}
Indeed, $w(z_{0},t_{0})\leq 0$ implies $w\leq0$ and so
\[
u-v \leq e^{KT}\|u_{0}-v_{0}\|_{C^{0}(\Omega)}.
\]
Interchanging $u$ and $v$, would then produce the required estimate.

\medskip

We show \eqref{stability eqn 1}. If $(z_{0},t_{0})\in\de^{*}Q_{T}$, then either
\[
w(z_{0},t_{0}) =
\begin{cases}
\ u_{0}(z_{0})-v_{0}(z_{0}) \ & \mbox{if $(z_{0},t_{0})\in\Omega_{0} $}, \text{ or } \\
\ 0 \ & \mbox{if $(z_{0},t_{0})\in\bar{\Gamma}_{T}$}.
\end{cases}
\]
In both cases, \eqref{stability eqn 1} follows. If $(z_{0},t_{0})\notin\de^{*}Q_{T}$, then at $(z_{0},t_{0})$, we have
\[
\frac{\de w}{\de t} \geq 0, \ \
\ddbar w \leq 0.
\]
Thus,
\[
-K e^{-Kt}(u-v)+e^{-Kt}(u_{t}-v_{t}) \geq 0, \ \
e^{K t}(\ddbar u-\ddbar v) \leq 0
\]
and so
\[
(u_{t}-v_{t}) \geq K(u-v), \ \
\ddbar u \leq \ddbar v.
\]
From $\ddbar u \leq \ddbar v$, we have $F(u)\leq F(v)$. We then compute
\[
\begin{split}
K(u-v) \leq {} & (u_{t}-v_{t}) = F(u)-F(v)-f(z,t,u)+f(z,t,v) \leq -f(z,t,u)+f(z,t,v).
\end{split}
\]
It follows that
\[
f(z,t,u)+Ku \leq f(z,t,v)+K v \ \ \text{at $(z_{0},t_{0})$}.
\]
Now since $|f_{s}|\leq K$ in $[-M_{0},M_{0}]$, the function $f(z,t,s)+Ks$ is increasing in $s\in[-M_{0},M_{0}]$. The above now implies that $u\leq v$ at $(z_{0},t_{0})$ and so $w(z_{0},t_{0})\leq0$, as required.
\end{proof}

\section{The sublinear case}\label{sublinear section}

In this section, we will prove Theorem \ref{Sublinear Theorem}; for the reader's convenience, we recall the statement below:

\begin{theorem}[Sublinear Growth]\label{Sublinear Section 4}
	 Let $\psi\in C(\bar{\Omega\times\RR_-})\cap C^{1,1}(\bar\Omega\times\RR_-)$. Suppose that $\psi(z,x)>0$ for $x<0$, and that
	 \begin{align}
		&\lim_{x\>0^-}\frac{\psi(z,x)}{|x|} > \lambda_1 \quad \text{uniformly for } z\in \ov{\Omega}, \text{ and } \label{eq:sublinear growth cond go to 0-}\\
	 	&\lim_{x\rightarrow-\infty}\frac{\psi(z,x)}{|x|} <\lambda_1\quad \text{uniformly for } z\in\ov{\Omega}. \label{eq:sublinear condition}
	 \end{align}
 Then there exists a non-zero solution $u\in  C^{1,1}(\bar\Omega)\cap C^\infty(\Omega) \cap \cH(\Omega)$ to \eqref{eq:initial equation} which minimizes $J$ over $\mathcal{H}(\Omega)$.
\end{theorem}

We start by showing that the sublinear growth condition \eqref{eq:sublinear condition} implies that $J$ is proper with respect to $E$.
\begin{proposition}\label{sublinear_proper}
There exists $\theta, K > 0$ such that
\[
J(u) \geq \theta E(u) - K
\]
for all $u\in\mathcal{H}(\Omega)$.
\end{proposition}
\begin{proof}
By \eqref{eq:sublinear condition}, there exists constants $K , \theta > 0$ such that:
\begin{align}\label{eq:upper of psi}
	\psi(z,x)^n\leq K + (1-\theta)\lambda_1^n|x|^{n}
\end{align}
for all $x \leq 0$. Similarly, up to possibly increasing $K$, we have:
\begin{align}\label{eq:upper of Psi}
	\Psi(z,x)\leq K +\frac{(1-\theta)\lambda_1^n}{n+1}|x|^{n+1}
\end{align}
for all $x \leq 0$. By \eqref{eq:Rayleigh Quotient}, we then see
\begin{align}
	J(u)
	\geq& E(u)-K\Vol(\Omega)-(1-\theta)\lambda_1^n\frac{1}{n+1}\int_\Omega(-u)^{n+1}\omega^n
	\geq\theta E(u)-K\Vol(\Omega),
\end{align}
as desired.
\end{proof}
In particular, $J$ is always bounded from below and $J(u)\>\infty$ as $E(u)\>\infty$. We deduce from this {\it a priori} $L^\infty$-bounds on solutions to \eqref{eq:initial equation} under the sublinear condition \eqref{eq:sublinear condition}:

\begin{lemma} \label{lemm:uniform_Linfty_sublinear}
	Consider \eqref{eq:initial equation} under the condition \eqref{eq:sublinear condition}, and let $u\in \cH(\Omega)$ be a solution of \eqref{eq:initial equation} satisfying $J(u)\leq A_0$.
	Then there exists a constant $C = C(n, K,\theta,A_0, \Omega)$ such that
	\begin{align*}
		\sup_{z\in\Omega}|u(z)|\leq C.
	\end{align*}
\end{lemma}

\begin{proof}
	By Proposition \ref{sublinear_proper}, one has
	\begin{align*}
		E(u)\leq C,
	\end{align*}
    where $C=C(A_0, \theta, K)$. Let $p = 1+\frac{1}{n}$. By \eqref{eq:upper of psi} and \eqref{eq:Sobolev ineq}, one obtains
	\begin{align*}
		\int_\Omega\psi(z,u)^{np}\omega^n       \leq \int_\Omega(K +(1-\theta)\lambda_1|u|)^{np}    \leq C.
	\end{align*}
    By Ko\l odziej's $L^\infty$-estimate \cite{Kol98}, we are done.
\end{proof}

\begin{remark}
Actually, by using the uniform H\"older estimate for the complex Monge--Amp\`ere operator shown in \cite{HZ26}, one can remove the assumption that $J(u)\leq A_0$ from Lemma~\ref{lemm:uniform_Linfty_sublinear}. The proof is by contradiction, following roughly \cite[Lemma~6.1]{CW01}; indeed, if the conclusion of the lemma failed to hold, then there would exist a sequence of right-hand sides $\{\psi_m\}_{m=1}^\infty$, each satisfying \eqref{eq:upper of psi}, and such that the equation \eqref{eq:initial equation} with $\psi=\psi_m$ has a solution $u_m\in\cH(\Omega,\omega)$ with
	\begin{align*}
		M_m:=\sup_\Omega|u_m|\>\infty
	\end{align*}
	as $m\>\infty$.	Setting $v_m:=\frac{u_m}{M_m}$, we see that $v_m$ is uniformly bounded and solves
	\begin{align*}
		\begin{cases}
		(\ddc v_m)^n=\frac{\psi^n(z, u_m)}{M_m^n}\omega^n, &\text{ in }\Omega, \\
		v_m=0, &\text{ on }\p\Omega.
	\end{cases}
	\end{align*}
    By \eqref{eq:upper of psi}, we have
    \begin{align*}
	(\ddc v_m)^n \leq \left(\frac{K}{M_m^n}+(1-\theta)(-\lambda_1 v_m)^n\right)\omega^n.
	\end{align*}
    By \cite[Theorem~1.1]{HZ26}, up to taking a subsequence, the $v_m$ converge uniformly on $\ov{\Omega}$ to some $v\in\PSH(\Omega)$, which is necessarily non-zero. Since the complex Monge--Amp\`ere operator is weakly continuous under uniform convergence, we see
	\begin{align*}
		(\ddc v)^n 	\leq (1-\theta)\lambda_1^n (-v)^n\omega^n;
	\end{align*}
	but then Proposition \ref{supersolution prop} implies $v = 0$, a contradiction.
\end{remark}
\medskip

To show higher regularity of solutions to \eqref{eq:initial equation}, we will use a flow argument; this is especially convenient since $J$ will not have good convexity properties in general.

Following \cite{CW01}, we consider the flow:
\begin{align}
\label{eq:neg gradient flow}
	\begin{cases}
		\log\MA(u) -u_t= \log \psi^n(z, u) & \text{in }\Omega\times(0,\infty)  \\
		u=0 & \text{on }\p\Omega\times[0,\infty)  \\
		u=u_0 &\text{ on } \{t=0\},
	\end{cases}
\end{align}
where we assume $u_0\in \mathcal{H}(\Omega)$ satisfies the compatibility condition \eqref{compatibility condition}. This is a negative gradient flow for $J$, since, if $u = u(z,t)$ is a smooth solution, then
\begin{align}
	\frac{d}{dt} J(u(z,t)) =&\int_\Omega(-u_t)\left(\MA(u) -\psi^n(u)\right)\omega^n \nonumber\\
	=&-\int_\Omega\Big(\log\MA(u)-\log\psi^n(u)\Big)\Big(\MA(u)-\psi^n(u)\Big)\omega^n \leq0.\label{eq:negative derivative along flow}
\end{align}

We will first show Theorem \ref{Sublinear Section 4} under the additional assumption that $\psi \geq \psi_0 > 0$, by using the estimates from Section \ref{mu flow section}:
\begin{theorem} \label{thm:existence_sublinear_spositive}
	Let $(\Omega,\omega)$ be a strongly pseudoconvex K\"ahler manifold, $\psi\in C^{1,1}(\bar{\Omega\times\RR^-})$, $\psi\geq\psi_0>0$ and suppose that  \eqref{eq:sublinear condition} holds.
	Then \eqref{eq:initial equation} has a solution $u\in C^{\infty}(\bar\Omega)\cap\cH(\Omega)$, which is moreover an absolute minimizer of $J$ over $\mathcal{H}(\Omega)$.
\end{theorem}
\begin{proof}
	We approximate $\psi$ by right-hand sides which are bounded from above; for each $m \geq 0$, choose $\psi_m(z,x)\in  C^{1,1} (\bar\Omega\times\RR)$ such that
    \begin{itemize}
    \item $\psi_m\nearrow \psi$ pointwise as $m\rightarrow\infty$,
    \item $\psi_m(z,x)=\psi(z,x)$ when $|x|\leq m$, and
    \item $\psi_m(z,x)$ is independent of $x$ when $|x|\geq 2m$.
	\end{itemize}
    Define functionals $J_m$ on $\mathcal{H}(\Omega)$ by
	\begin{align*}
		J_m(u)\coloneqq E(u)-\int_\Omega\Psi_m(z,u) \omega^n,
	\end{align*}
	where $\Psi_m(z,x)\coloneqq\int_x^0\psi_m^n(z,s)ds$. Note that $J \leq J_m$.

    Let $0 < a < \frac{1}{4}$ be arbitrary. Choose $u_0\in\cH(\Omega)$ such that $J_m(u_0) < a +\inf_{\cH(\Omega)}J_m$. By adding $\e\rho$ to $u_0$, for some $\e > 0$ sufficiently small, we may additionally assume that $\ddbar u_0 \geq\tilde{\epsilon}\omega$ and $J_m(u_0) < 2a +\inf_{\cH(\Omega)}J_m$, for some small $\tilde{\e}>0$.

    We modify $u_0$ further so that it satisfies the compatibility condition \eqref{compatibility condition}. Pick $\delta > 0$ sufficiently small and set $\Omega_\delta\coloneqq\{z\in\Omega~|~\dist_\omega(z,\Omega)>\delta\}$. Let $g\in C^{1,1}(\bar\Omega)$, $g > 0$, be such that $g=\MA(u_0)$ in $\Omega_\delta$ and $g =\log\psi_m^n(z,0)=\log\psi^n(z,0)$ on $\Omega\setminus \Omega_{\delta/2}$. Let $\tilde{u}_0\in\cH(\Omega)$ be the solution to
    \[
    \begin{cases}
    \MA(u)=g\\
    u|_{\p \Omega} = 0;
    \end{cases}
    \]
    clearly $\tilde{u}_0$ satisfies \eqref{compatibility condition}. Further, by choosing $\delta > 0$ appropriately small, we can ensure that $J_m(\tilde u_0) < 4a +\inf_{\cH(\Omega)}J_m$, using stability of the Monge--Amp\`ere operator and continuity of $J$ along uniformly convergent sequences. We replace $u_0$ with $\tilde u_0$ for the rest of the proof.

Now we may apply Theorem~\ref{thm:estimate of flow} (with $p = \infty$ and $f = \log \psi_m^n$) to find a smooth solution $u=u_{m,a}$ to the flow
	\begin{equation}
	\begin{cases}
		\log\MA(u)-u_t=\log\psi_m^n(z,u) & \text{in }Q=\Omega\times(0,\infty)  \\
		u=0 & \text{on }\p\Omega\times[0,\infty)  \\
		u(z,0)=u_0, & u(\cdot,t)\in\cH(\Omega).
	\end{cases}
\end{equation}
By \eqref{eq:negative derivative along flow}, we have for each $t > 0$ that:
\[
\inf_{\mathcal{H}(\Omega)} J_m \leq J_m(u(z, t)) \leq J_m(u_0),
\]
so there must exist some sequence of times $t_j\rightarrow\infty$ such that:
\begin{align*}
&\frac{d}{dt} J_m(u(z,t_j))  \\
=& -\int_\Omega\Big(\log\MA(u(z,t_j))-\log\psi^n(u(z, t_j))\Big)\Big(\MA(u(z, t_j))-\psi^n(u(z, t_j))\Big)\omega^n\rightarrow 0.
\end{align*}
Since Theorem~\ref{thm:estimate of flow} further implies that
\begin{align*}
	\|u\|_{C^{2,1}(\bar\Omega\times\RR)}\leq C,
\end{align*}
we can choose a subsequence of the $t_j$ such that $u(z, t_j)$ converges uniformly to some $w_{m,a} \in \mathcal{H}(\Omega)$. It follows then from continuity properties of the Monge--Amp\`ere operator that $w_{m,a}$ satisfies:
\begin{align*}
	\begin{cases}
		\MA(w_{m,a})=\psi_m^n(z,w_{m,a}), &\text{ in }\Omega, \\
		w_{m,a}=0, &\text{ on }\p\Omega,
	\end{cases}
\end{align*}
Since $J_m$ is also easily shown to be continuous under uniform convergence, we further have:
\[
J_m(w_{m,a}) \leq J_m(u_0) \leq 4a + \inf_{\mathcal{H}(\Omega)} J_m \leq 1 + \inf_{\mathcal{H}(\Omega)} J_0,
\]
which is a uniform constant. Further, arguing exactly as in Lemma \ref{lemm:uniform_Linfty_sublinear}, we see that $w_{m,a}$ is uniformly bounded, independent of $m$ and $a$ -- indeed, by Proposition \ref{sublinear_proper}, we have:
\[
E(w_{m,a}) \leq \theta^{-1}(J(w_{m,a}) + K) \leq \theta^{-1}(J_m(w_{m,a}) + K) \leq C;
\]
the rest of the argument only depends on an upper bound for $\psi_m \leq \psi$.

Choosing $m > \Vert w\Vert_{L^\infty}$ now implies that $\psi_m(z, w_{m,a}) = \psi(z, w_{m,a})$, so that $w_a := w_{m,a}$ solves \eqref{eq:initial equation}. Since $\psi \geq \psi_0 > 0$, Theorem~\ref{thm:C3alpha-estimate} now implies that $w_a$ is uniformly bounded in $C^{3,\alpha}(\ov{\Omega})$ independent of $a$; we can thus take the limit $a \rightarrow 0$ to finish.
\end{proof}

We now show Theorem \ref{Sublinear Section 4}:

\begin{proof}[Proof of Theorem \ref{Sublinear Section 4}]
Let $\epsilon \in (0, 1)$ and consider:
\begin{equation} \label{eq:approx}
    \begin{cases}
        \MA(u_\epsilon) = (\psi^n(z, u_\epsilon)+\epsilon)\omega^n  & \text{in } \Omega, \\
        u_\epsilon = 0 & \text{on } \p\Omega.
    \end{cases}
\end{equation}
By Theorem~\ref{thm:existence_sublinear_spositive}, there exists a solution $u_\epsilon\in C^{3,\alpha}(\bar\Omega)\cap\cH(\Omega)$ to \eqref{eq:approx} which minimizes the functional $J_\e$:
\[
	J_\epsilon(u)\coloneqq E(u)-\int_\Omega\Psi_\epsilon(z,u) \omega^n,\quad u\in\mathcal{H}(\Omega).
\]
Here, we have defined $\Psi_\epsilon(z,x)\coloneqq\int_x^0(\psi^n(z,s)+\epsilon)ds=\Psi(z,x)+\epsilon|x|$. Note that $J_\e \leq J$, so that $J_\e(u_\e) = \inf_{\mathcal{H}(\Omega)} J_\e \leq \inf_{\mathcal{H}(\Omega)} J$.

Since it is clear that:
\[
\Psi_\e(z, x) \leq K + \frac{(1-\frac{\theta}{2})\lambda_1^n}{n+1}|x|^{n+1}
\]
for $\e$ sufficiently small and $\theta$ as in Proposition \ref{sublinear_proper}, we can repeat the arguments of Lemma \ref{lemm:uniform_Linfty_sublinear} to deduce a uniform $L^\infty$-bound on $u_\e$, independent of $\e$.

It follows now from Theorem \ref{thm:C3alpha-estimate} that $u_\e$ is uniformly bounded in $C^2(\ov{\Omega})$, independent of $\e$; by the Arzel\`a--Ascoli theorem, there exists a subsequence $u_{\epsilon_j}$ converging uniformly on $\bar{\Omega}$ to $u \in C^{1,1}(\bar{\Omega})$ which solves \eqref{eq:initial equation}. Since $\psi(z, x) > 0$ for $x < 0$, the interior regularity theory of Caffarelli--Kohn--Nirenberg--Spruck \cite{CKNS85} and Evans--Krylov \cite{Kry82} implies that $u \in C^\infty(\Omega)$.

Finally, we observe that condition \eqref{eq:sublinear growth cond go to 0-} implies that 0 is not a minimizer, since it implies that $J(c u_1)$ is decreasing in $c$ for $c > 0$ sufficiently small. Thus, the solution $u$ cannot be zero.
\end{proof}


\section{The superlinear case} \label{superlinear section}

In this section, we will prove Theorem \ref{superlinear_theorem_intro}, reproduced here for the reader's convenience:

\begin{theorem}[Superlinear]\label{superlinear_theorem}
	Let $\psi\in C(\bar{\Omega\times\RR_-})\cap C^{1,1}(\bar\Omega\times\RR_-)$. Suppose that $\psi(z, x) > 0$ for $x < 0$, and that there exists constants $0 < \sigma < 2n$, $0 < \theta < 1$, and $M > 0$ such that:
\begin{align}
&\lim_{x\>0^-}\frac{\psi(z,x)}{|x|}<\lambda_1 \quad \text{uniformly for } z\in \ov{\Omega} \label{eq:growth cond go to 0-}\\
&\lim_{x\rightarrow -\infty} \psi(z, x)e^{-\frac{\sigma}{n} \abs{x}^{1+\frac{1}{n}}} <\infty \quad \text{uniformly for }z\in\ov{\Omega}\label{eq:exp growth}\\
&\int_x^0\psi^n(z,s)ds \leq\frac{1-\theta}{n+1}|x|\psi^n(z,x) \quad \text{ for all $x \leq -M$ and $z\in\ov{\Omega}$.} \label{eq:integration growth upper}.
\end{align}
Then there exists a solution $u\in C^{1,1}(\bar\Omega) \cap C^{\infty}(\Omega) \cap \cH(\Omega)$ of \eqref{eq:initial equation}.
\end{theorem}

As remarked in the Introduction, by Gr\"onwall's inequality, condition \eqref{eq:integration growth upper} implies that $\psi(z, x)$ is growing ``superlinearly" in the $x$-variable. Specifically, it implies:
\begin{equation}\label{eq:superlinear condition}
\lim_{x\>-\infty}\frac{\psi(z,x)}{|x|}>\lambda_1 \quad \text{uniformly for } z\in \ov{\Omega},
\end{equation}
which mirrors one of the assumptions in Theorem \ref{Sublinear Theorem}.

Recall also that $J$ is not proper in the superlinear case, and its critical points will be saddles. We thus employ a mountain-pass type argument to show the existence of solutions. We again largely follow the proof of Chou--Wang \cite{CW01}.

\begin{proof} We break the proof into several steps.

\smallskip

\noindent{\bf Step 1:} We start by approximating our problem, both by adding a small $\delta$ to $\psi$ to ensure non-degeneracy, but also by curtailing the growth of $\psi$ at $-\infty$ (if necessary) to be at most polynomial.

We start by fixing some notation. First, we choose $A > 0$ such that, up to possibly increasing $M \geq 1$ and shrinking $\theta$, we have by \eqref{eq:superlinear condition} and \eqref{eq:exp growth}:
\begin{equation}\label{large x bounds}
(1 + \theta)\lambda_1^n\abs{x}^n \leq \psi^n(z, x) \leq Ae^{\sigma\abs{x}^{1+\frac{1}{n}}}\text{ for all $x\leq -M$ and $z\in\ov{\Omega}$.}
\end{equation}
Further, by \eqref{eq:growth cond go to 0-}, we can assume:
\begin{equation}\label{small x bounds}
\psi^n(z,x) \leq (1 - \theta)\lambda_1^n \abs{x}^n \text{ for all $-M^{-1} \leq x \leq 0$ and $z\in \ov{\Omega}$.}
\end{equation}

In the arguments below, we use $K$ to denote a uniform constant, whose exact value may change from line to line, but which can always be computed solely in terms of $(\Omega,\omega)$, $\psi$, and the fixed integer $p$ to be chosen momentarily. We will typically use $C$ to denote a constant which is not necessarily uniform (and whose value may also change from line to line).\\

We begin by restricting the growth of $\psi$ to be polynomial. Let $m \geq M$. If there exists an integer $p \in (n+1, \infty)$ such that:
\begin{equation} \label{eq:stronger growth cond}
\lim_{x\rightarrow-\infty}\frac{\psi(z, x)}{\abs{x}^{p/n}} = 0 \quad\text{ uniformly for }z\in \bar\Omega,
\end{equation}
then set $\psi_m = \psi$ for all $m$. Otherwise, we fix an integer $p \in (n+1, \infty)$ such that $\frac{n+1}{p} < \frac{\theta}{2}$ and choose $\psi_m \in C^{1,1}(\ov{\Omega}\times \RR)$ so that:
\begin{itemize}
\item $\psi_m^n(z,x) = \psi^n(z,x)$ for $\abs{x} \leq m$;
\item $\psi_m^n(z,x) = K_m\abs{x}^{p-1}$ for $\abs{x} \geq m + \delta_m$;
\item $\psi^n(z,-m) - 1 \leq \psi_m^n(z,x) \leq K_m \abs{m+\delta_m}^{p-1}$ for $m < \abs{x} < m+\delta_m$; and
\item $\psi_m^n(z,x)$ still satisfies \eqref{large x bounds}.
\end{itemize}
Here, we set $B \coloneqq \sup_{z\in\ov{\Omega}} \psi^n(z, -m) > 0$ and choose the constants $K_m, \delta_m > 0$ to be
\[
\delta_m = (B+1)^{-1} \quad\text{ and }\quad K_m = (B + 1)m^{1-p}.
\]
It is the clear that $\psi_m(z,x)$ also satisfies both \eqref{small x bounds} and \eqref{eq:stronger growth cond}. In addition, $\psi_m$ satisfies \eqref{eq:integration growth upper} with $\theta$ replaced by $\frac{\theta}{2}$, since for $x < -(m  +\delta_m)$ and $m$ sufficiently large:
\begin{align*}
\int_x^0 \psi_m^n(z, s) ds &\leq \int_{-m}^0 \psi^n(z, s) ds + \int^{-m}_{-m-\delta_m} \psi_m^n(z, s) ds + \int^{-m-\delta_m}_x K_m\abs{s}^{p-1} ds\\
&\leq \frac{1-\theta}{n+1} m\psi^n(z, -m) + 1 + K_m\frac{\abs{x}^{p}}{p} - K_m\frac{(m + \delta_m)^{p}}{p}\\
&\leq \frac{1-(\theta/2)}{n+1} \abs{x} \psi_m(z, x);
\end{align*}
for $-(m+\delta_m) < x < -m$ the inequality is immediate.\\

Now, for each $\delta \geq 0$ sufficiently small, we choose a nonnegative, smooth function  $\eta_\delta$, such that $0\leq\eta_\delta\leq 1$, $\eta_\delta(z)=0$ when $\dist_\omega(z,\p\Omega)\leq\delta$, and $\eta_\delta(z)=1$ when $\dist_\omega(z,\p\Omega) \geq 2\delta$. Note that we set $\eta_0\equiv1$. We then define:
\begin{align*}
	\psi_{m,\delta}^n(z,x) :=\eta_\delta(z)\psi^n_m(z,x)+\delta^2.
\end{align*}

Finally, we define:
\begin{align*}
&\Psi_m(z, x) := \int_x^0 \psi_m^n(z, s)\, ds\quad\quad \Psi_{m, \delta}(z, x) := \int_x^0 \psi_{m, \delta}^n(z, s)\, ds
\end{align*}
and
\begin{align*}
&J_m := E - \int_\Omega \Psi_m \omega^n \quad\quad J_{m, \delta} := E - \int_\Omega \Psi_{m, \delta} \omega^n.
\end{align*}
It is easy to see that $J_{m,\delta}$ is an Euler--Lagrange functional for the non-degenerate equation
\begin{equation}\label{eq:perturbed equation}
	\begin{cases}
		(\ddc u)^n=\psi_{m,\delta}^n(z,u)\omega^n, &\text{ in }\Omega, \\
		u=0, &\text{ on }\p\Omega.
	\end{cases}
\end{equation}

\smallskip

\noindent{\bf Step 2:} Using the polynomial growth condition \eqref{eq:stronger growth cond}, we show that, for each $\delta \geq 0$, the $J_{m,\delta}$ admit a mountain-pass-like point, and that the critical values of these points converge as $\delta \rightarrow 0$ to the critical value of $J_m$, which is strictly positive.

Start by combining \eqref{small x bounds} and \eqref{eq:stronger growth cond} to see that there exists another large constant $K_m >0$ (which again depends only uniform quantities and $m$) such that
\begin{align}\label{eq:Psi_growth_superlinear}
	\Psi_m(z,u) \leq \frac{(1-\theta)\lambda_1^n}{n+1}|u|^{n+1} +K_m |u|^{p+1}.
\end{align}
This will force $J_m$ to have a saddle-like point. From the above:
\begin{align*}
	J_m(u) \geq& E(u)-(1-\theta)\frac{\lambda_1^n}{n+1}\int_\Omega|u|^{n+1}\omega^n +K_m\int_\Omega|u|^{p+1}\omega^n   \\
	\geq& E(u)-(1-\theta)E(u)-K_mE(u)^{\frac{p+1}{n+1}} \\
	=&\theta E(u)-K_mE(u)^{\frac{p+1}{n+1}},
\end{align*}
using the Rayleigh quotient formula \eqref{eq:Rayleigh Quotient} and the energy inequality \eqref{eq:Sobolev ineq}. Thus, there exists a small constant $k_m > 0$, again depending only on uniform quantities and $m$, such that for all $u \in \cH(\Omega)$ with $E(u) = k_m^{n+1}$, we have
\begin{align*}
	J_m(u)\geq \frac{\theta k_m^{n+1}}{n+1} > 0.
\end{align*}
On the other hand, we have $J_m(0) = 0$ while $J_m(Ku_1) \leq -1$ for all $m \geq M$ and $K >0$ sufficiently large, by the lower bound in \eqref{large x bounds} (recall that $u_1$ is the first eigenfunction). It follows that $J_m$ should have a saddle-like point somewhere between these two extremes.\\

It will be convenient to regularize $Ku_1$ slightly -- let $v_1\in \mathcal{H}(\Omega)\cap C^\infty(\ov{\Omega})$ be sufficiently close to $Ku_1$ so that $J_m(v_1) < -\frac{1}{2}$ for all $m$ and $\MA(v_1) > 0$ on all of $\ov{\Omega}$. We also fix $v_0 := \e v_1$, for some small $\e > 0$ such that $E(v_0) < k_m^{n+1}$ and $J_m(v_0) < \frac{\theta k_m^{n+1}}{2(n+1)}$. \\

We now define:
\begin{align*}
	\Gamma_m  \coloneqq\left\{\gamma\in C([0,1], \cH(\Omega)\cap C^{3,1}(\bar\Omega))~|~ \gamma(0)=v_0, \gamma(1) = v_1, \MA(\gamma(s))>0 \text{ on }\bar\Omega\right\},
\end{align*}
and
\begin{equation} \label{eq:min-max const zero}
	c_m =\inf_{\gamma\in\Gamma_m}\sup_{s\in[0,1]}J_m(\gamma(s)).
\end{equation}
Observe that $c_m > 0$, since $E$ is continuous along any path $\gamma\in \Gamma$.

We also define:
\begin{equation}\label{eq:min-max const}
	c_{m,\delta} =\inf_{\gamma\in\Gamma_m}\sup_{s\in[0,1]}J_{m, \delta}(\gamma(s)),
\end{equation}
Since any fixed $\gamma(s)\in \Gamma_m$ is uniformly bounded, $J_{m,s}(\gamma(s))$ converges uniformly in $s$ to $J_m(\gamma(s))$ as $\delta\rightarrow 0$. It follows that:
\begin{align*}
	c_m\geq \bar{\lim_{\delta\>0}}c_{m,\delta}.
\end{align*}
Since $\psi_{m,\delta}^n(z,x)\leq \psi_m^n(z,x)+\delta^2$, we see that
\begin{align*}
	c_m=\lim_{\delta\>0}c_{m,\delta}.
\end{align*}
It follows that $c_{m,\delta}\geq\frac{1}{2}c_m \geq \frac{\theta k_m^{n+1}}{2(n+1)} >0$ for all $\delta$ sufficiently small.

Note also that the $c_{m,\delta}$ are uniformly bounded from above, by $E(K v_1)$ for example. \\

\noindent{\bf Step 3:} We now set up an appropriate path of flows, which will eventually converge to a solution $u_{m,\delta}$ to \eqref{eq:perturbed equation} with $J_{m,\delta}(u_{m,\delta}) = c_{m,\delta}$; convergence will be established after proving some further estimates in the next two steps. To simplify notation, we fix $m$ and omit it from most of the subscripts below.

Let $0 < a < \frac{1}{4}c_m$, and choose $\gamma\in\Gamma$ with
\begin{align}\label{for reference}
	\sup_{s\in[0,1]}J_\delta(\gamma(s))\leq c_\delta+ a.
\end{align}
We construct a smooth function $\mu$ as in Lemma \ref{construction of mu} and consider the parabolic equation
\begin{equation}\label{eq:perturbed flow}
	\begin{cases}
		\mu(\MA(u))-u_t=\mu(\psi_\delta^n(z,u)) & \text{in }Q=\Omega\times(0,\infty)  \\
		u(z,0)=\gamma(s), & u(\cdot,t)\in\cH(\Omega,\omega).
	\end{cases}
\end{equation}
As done in the proof of Theorem~\ref{thm:existence_sublinear_spositive}, we can modify $\gamma$ to a new path in $\Gamma$, which still satisfies \eqref{for reference}, and has each $\gamma(s)\in\Gamma$ satisfying the compatibility condition \eqref{compatibility condition} from Theorem~\ref{thm:estimate of flow}.

By \eqref{eq:stronger growth cond}, one has
\begin{align*}
	|\mu(\psi_\delta^n(z,x))|\leq C(1+|x|)
\end{align*}
for large $x$. It follows from Theorem~\ref{thm:estimate of flow} that there is a solution $u^s(\cdot,t)$ of \eqref{eq:perturbed flow}, for each $s\in[0,1]$ and all $t\geq 0$; these solutions are not under any uniform control however. In particular, it is not obvious that the $u^s$ should remain bounded as $t\rightarrow\infty$. \\

In the later steps, we will show that this indeed the case for those $u^s$ which remain near the critical point for all large time. We finish this step by showing that such $s$ exist.

For each $t > 0$, let $I_t:=\{s\in[0,1]~|~ J_\delta(u^s(\cdot,t))\geq c_\delta-a\}$. By definition, $I_t$ is a closed subset of $[0,1]$, and $I_t\subseteq I_{t'}$ for any $t\geq t'$. Set
\begin{align*}
	I_\infty:=\bigcap_{t\geq0} I_t.
\end{align*}
We see that $I_\infty$ is non-empty; if it were, there would exist some fixed $t >0$ such that $I_t=\emptyset$, i.e., $J_\delta(u^s(\cdot,t))\leq c_\delta-a$ for all $s\in[0,1]$. This would imply that
\begin{align*}
	\sup_{u\in\tilde{\gamma}^t} J_\delta(u) \leq c_\delta-a,
\end{align*}
where $\tilde{\gamma}^t\in \Gamma$ is the path defined by concatenating the three continuous paths:
\[
s \to u^0(z, s),\ \  s\to u^s(z, t),\ \text{ and }s\to u^1(z, s);
\]
continuity of the middle path follows from the stability estimate in Lemma \ref{stability lemma}. This would then contradict the definition of $c_\delta$, so that $I_\infty$ cannot be empty.\\

\noindent{\bf Step 4:} Fix $s_0 \in I_\infty$ for the rest of the proof. The goal of this step is to establish uniform energy control over $u^{s_0}$ for a sequence of times going to infinity.

Specifically, we define the set of good times:
\[
T_0 \coloneqq\left \{ t > 0~\middle|~ \frac{d}{dt}J_\delta(u^{s_0}(z, t)) \geq -a\right\}.
\]
Since $c_\delta - a \leq J_\delta(u^{s_0}(z, t)) \leq c_\delta + a$ for all $t$, we see that the complement $T_0^c$ must have measure $< 2$.

Fix some $t \in T_0$; we suppress $t$ in the notation below. In addition, we set:
\begin{align*}
	\alpha:=\Big(\MA(u^{s_0}(z,t))\Big)^{\frac{1}{p}} \quad\text{ and }\quad
	\beta:=(\psi_\delta^n(z,u^{s_0}(z,t)))^{\frac{1}{p}}.
\end{align*}
We then aim to show that there exists a uniform $K > 0$ such that both:
\begin{equation}\label{key uniform estimate}
E(u^{s_0}(z, t)) = \frac{1}{n+1}\int_\Omega (-u^{s_0})\alpha^p \leq K\quad \text{ and }\quad \int_\Omega (-u^{s_0})\beta^p\omega^n \leq K
\end{equation}
for all $\delta$ and $a$ uniformly small. We may assume that both $E(u^{s_0}), \int_\Omega(-u^{s_0})\beta^p\omega^n \geq 1$ without loss of generality.

Start by observing that \eqref{eq:integration growth upper} and the definition of $\psi_\delta$ implies that for $z\in \Omega$ where $u^{s_0}(z, t) < -M$, we have:
\begin{align*}
	\Psi_\delta(z,u^{s_0})
	=&-\delta^2u^{s_0}+\Psi(z,u^{s_0})\eta_\delta(z) \\
	\leq&\delta^2|u^{s_0}|+\frac{1-\theta}{n+1}|u^{s_0}|\psi^n(z,u^{s_0})\eta_\delta(z) \\
	=&\delta^2\frac{n-\theta}{n+1}|u^{s_0}|+\frac{1-\theta}{n+1}|u^{s_0}|\psi^n_\delta(z,u^{s_0}).
\end{align*}
It follows now from \eqref{eq:Sobolev ineq} and the uniform upper bound for $c_\delta$ (and hence $a$) that:
\begin{align}
	E(u^{s_0})\leq& J(u^{s_0}) + \int_\Omega\Psi_\delta(z,u^{s_0})\omega^n \nonumber \\
	\leq& (c_\delta+a) + \int_{\{u^{s_0} > -M\}} \Psi_\delta(z, u^{s_0})\omega^n +  \delta^2\frac{n+\theta}{n+1}\int_\Omega|u^{s_0}|\omega^n +\frac{1-\theta}{n+1}\int_\Omega|u^{s_0}|\beta^p\omega^n \nonumber \\
	\leq&K + \delta^2KE(u^{s_0})^{\frac{1}{n+1}} +\frac{1-\theta}{n+1}\int_\Omega|u^{s_0}|\beta^p\omega^n	\label{eq:E_k leq Psi_delta integ +c}\\
	\leq&K + \delta^2KE(u^{s_0}) +\frac{1-\theta}{n+1}\int_\Omega|u^{s_0}|\beta^p\omega^n. \nonumber
\end{align}
It follows that, for $\delta$ sufficiently small, we have:
\begin{equation}\label{compare eqn}
E(u^{s_0}) \leq K \int_\Omega\abs{u^{s_0}}\beta^p\omega^n .
\end{equation}
It will thus suffice to bound $\int_\Omega \abs{u^{s_0}}\beta^p\omega^n$. To do so, start by rewriting \eqref{eq:E_k leq Psi_delta integ +c} as:
\begin{equation}\label{rewrite eqn}
	\frac{\theta}{n+1}\int_\Omega \abs{u^{s_0}} \beta^p\omega^n \leq K + \delta^2K E(u^{s_0}) +\frac{1}{n+1}\int_\Omega|u^{s_0}|(\beta^p - \alpha^p)\omega^n.
\end{equation}
We will estimate the last term by using the good time condition $t\in T_0$. By \eqref{slope inequality}, we have
\[
	(r-s)(\mu(r)-\mu(s))\geq (r-s)(r^{\frac{1}{p}}-s^{\frac{1}{p}})
	\quad\text{ for }r, s>0,
\]
so \eqref{eq:negative derivative along flow} implies that
\begin{align*}
	a \geq -\frac{d}{dt}J_\delta(u^{s_0}(\cdot,t)) &=\int_\Omega(\alpha^p-\beta^p)(\mu(\alpha^p)-\mu(\beta^p))\omega^n \\
&\geq \int_\Omega(\alpha^p - \beta^p)(\alpha - \beta)\omega^n\\
& \geq \int_\Omega|\alpha-\beta|^{p+1}\omega^n.
\end{align*}
Now, by using the elementary inequality $\abs{x^p - y^p} \leq p \abs{x - y}(x^{p-1} + y^{p-1})$, for $x, y \geq 0$, and the H\"older inequality twice, we estimate:
\begin{align*}
	\left|\int_\Omega u^{s_0}(\alpha^p-\beta^p)\omega^n \right|
	\leq& K\int_\Omega|u^{s_0}|\cdot|\alpha-\beta|\cdot|\alpha^{p-1}+\beta^{p-1}|\omega^n  \\
	\leq&K\left(\int_\Omega|\alpha-\beta|^{p+1}\omega^n\right)^{\frac{1}{p+1}}
	   \left(\int_\Omega|u^{s_0}|^{\frac{p+1}{p}}(\alpha^{p-1}+\beta^{p-1})^{\frac{p+1}{p}}\omega^n\right)^{\frac{p}{p+1}} \\
	  \leq &K a^{\frac{1}{p+1}} \left(\int_\Omega|u^{s_0}|^{\frac{p+1}{p^2}}\cdot|u^{s_0}|^{\frac{p^2-1}{p^2}}(\alpha^{p-1}+\beta^{p-1})^{\frac{p+1}{p}}\omega^n\right)^{\frac{p}{p+1}} \\
	  \leq&K a^{\frac{1}{p+1}} \left(\int_\Omega|u^{s_0}|^{p+1}\omega^n\right)^{\frac{1}{p(p+1)}}  \left(\int_\Omega|u^{s_0}|(\alpha^{p-1}+\beta^{p-1})^{\frac{p}{p-1}}\omega^n\right)^{\frac{p-1}{p}}
\end{align*}
Using the power mean inequality, in the form $(x^{p-1} + y^{p-1})^{\frac{p}{p-1}} \leq 2^\frac{1}{p-1}\left(x^p + y^p\right)$, for $x,y \geq 0$, and then \eqref{eq:Sobolev ineq}, we conclude that:
\begin{align*}
\left|\int_\Omega u^{s_0}(\alpha^p-\beta^p)\omega^n \right|
	  \leq&K a^{\frac{1}{p+1}} \|u^{s_0}\|_{L^{p+1}}^{\frac{1}{p}} \left(\int_\Omega|u^{s_0}|(\alpha^{p}+\beta^{p})\omega^n\right)^{\frac{p-1}{p}} \\
\leq& Ka^{\frac{1}{p+1}}E(u^{s_0})^{\frac{1}{p(n+1)}} \left[\left(\int_\Omega|u^{s_0}|\alpha^{p}\omega^n\right)^{\frac{p-1}{p}}+\left(\int_\Omega|u^{s_0}|\beta^{p}\omega^n\right)^{\frac{p-1}{p}} \right].
\end{align*}
Using \eqref{compare eqn}, we can simplify this to:
\begin{align*}
\left|\int_\Omega u^{s_0}(\alpha^p-\beta^p)\omega^n \right| 
\leq&Ka^{\frac{1}{p+1}} \left(\int_\Omega|u^{s_0}|\beta^{p}\omega^n\right)^{1 - \frac{n}{p(n+1)}}.
\end{align*}
Plugging this into \eqref{rewrite eqn} and using \eqref{compare eqn} again, we arrive at:
\begin{equation}\label{a p eqn}
\frac{\theta}{n+1}\int_\Omega \abs{u^{s_0}} \beta^p\omega^n \leq K +\left(\delta^2 + a^{\frac{1}{p+1}}\right) K\int_\Omega|u^{s_0}|\beta^p\omega^n.
\end{equation}
Finally, by rearranging, we conclude that $\int_\Omega\abs{u^{s_0}} \psi(u^{s_0})\omega^n$ is uniformly bounded for all $\delta, a$ sufficiently small; by \eqref{compare eqn} again, we have shown \eqref{key uniform estimate}.\\

\noindent{\bf Step 5:} We will now use the uniform energy bounds \eqref{key uniform estimate} and the gradient estimates from Section \ref{mu flow section} to derive an $L^\infty$ bound for $u^{s_0}(z, t)$ which is independent of $t$ (we do not claim it to be independent of the other variables however).

Set $M_t := \sup_\Omega\abs{u^{s_0}(z,t)}$ and $N_t := \sup_{Q_t}\abs{u^{s_0}}$. By Propositions \ref{u t estimate} and \ref{gradient estimate}, we have the following bounds:
\begin{equation}\label{eq:u_bounds}
\quad |u_t^{s_0}(z,t)|\leq C_1(1+N_t), \quad\text{ and } |\nabla_zu^{s_0}(z,t)|\leq C_2(1+N_t^{p}),
\end{equation}
for $C_1, C_2 \geq 1$ which are independent of $t$.

We suppose for the sake of a contradiction that $M_t$ is unbounded as $t\rightarrow\infty$. Then there exists a sequence $t_j\rightarrow\infty$ such that
$M_{t_j}\rightarrow\infty$ and
\begin{align*}
	M_t\leq M_{t_j}\qquad \text{ for all }t\leq t_j.
\end{align*}
It follows that $M_{t_j} = N_{t_j}$. We can assume further that $N_{t_0} \geq \max\{C_1, C_2\}$ and that $t_{j+1} - t_{j} > 2$ for all $j \geq 1$.

We perturb the $t_j$ slightly to points in $T_0$, and such that they are still increasing in $j$ (e.g. decrease each $t_j$ until it is in $T_0$). Fix $0 < \e < \max\{2 , (1+C_1)^{-1}\}$. Then for all $t \in (t_j - \e, t_j)$, $j\geq 1$, we have:
\[
- \e (1+ C_1) N_{t_j} \leq -\int_t^{t_j} C_1(1 + N_s)\, ds \leq u^{s_0}(z, t_j) - u^{s_0}(z,t) \leq u^{s_0}(z, t_j) + M_t.
\]
Rearranging and taking the supremum over $z\in \Omega$ gives:
\[
C^{-1} N_{t_j} \leq M_t \quad  \text{ for all } t\in (t_j - \e, t_j)
\]
for some large $C > 0$, which is independent of $t$. Since $\abs{T_0^c} < 2$, we can find a sub-sequence of $j$ such that $T_0 \cap (t_j - \e, t_j)$ is non-empty; for each such $j$, we pick $\tau_j \in T_0\cap (t_j - \e, t_j)$, and note that:
\[
C^{-1} N_{\tau_j} \leq M_{\tau_j} \quad\text{ and }\quad M_{\tau_j} \rightarrow\infty\text{ as }j\rightarrow\infty.
\]

Now choose $z_j \in\Omega$ such that $u^{s_0}(z_j,\tau_j)=-M_{\tau_j}$. Pick some $0 < \e < \frac{C^{-p}}{4C_2}$ and set $R_j \coloneqq \e M_{\tau_j}^{1-p}$. It follows \eqref{eq:u_bounds} then that
\[
u^{s_0}(z, \tau_j) \leq -\frac{M_{\tau_j}}{2}\quad \text{for all } z\in B_{R_j}(z_j);
\]
indeed, we have:
\begin{align*}
|u^{s_0}(z,\tau_j)-u^{s_0}(z_j,\tau_j)|
\leq& \sup_z|\nabla_zu^{s_0}(z,\tau_j)| (2R_j)  \\
\leq& [C_2 (1+N_{\tau_j}^p)] (\e M_{\tau_j}^{1-p})
< 2 C_2 C^{p} (\e M_{\tau_j}) < \frac{1}{2} M_{\tau_j}.
\end{align*}
Now pick $q > 0$ such that $q > n(p-1)$. By \eqref{key uniform estimate} and \eqref{eq:Sobolev ineq}, we have
\begin{align*}
	\|u^{s_0}(z,\tau_j)\|_{L^q(B_{R_j}(y_j))}  \leq   K\|u^{s_0}(z,\tau_j)\|_{L^q(\Omega)}  \leq K E(u^{s_0}(z,\tau_j))^{\frac{1}{n+1}}\leq K.
\end{align*}
But on the other hand, we have
\begin{align*}
	\|u^{s_0}(z,\tau_j)\|^q_{L^q(B_{R_j}(y_j))} \geq C M_{\tau_j}^q R_j^n=CM_{\tau_j}^{q+(n-np)},
\end{align*}
which is a contradiction, since $M_{\tau_j}\rightarrow\infty$. We obtain $\sup_{z\in\ov{\Omega}} |u^{s_0}(z,t)|\leq C$ for all $t\geq0$.\\

\noindent{\bf Step 6:} We are now ready to take limits to produce a solution to the truncated problem, which will provide a solution to our initial problem \eqref{eq:initial equation} for all $m$ uniformly large.

We start by taking the limit as $t\rightarrow\infty$. Recall that $c_{m,\delta} - a \leq J_{m,\delta}(u^{s_0}(\cdot,t))\leq c_{m,\delta} + a$ for all $t \geq 0$. Thus, we can choose some sequence of times $t_j\rightarrow\infty$ such that
\begin{align*}
	\frac{d}{dt}J_{m,\delta}(u^{s_0}(\cdot,t))\>0
\end{align*}
By Theorem~\ref{thm:estimate of flow} and the $t$-uniform bound shown in Step 5, we can take a further subsequence such that $\{u^{s_0}(z, t_j) \}$ converges uniformly to some $u_{m,\delta}(z) \in C^{3}(\Omega)$. By \eqref{eq:negative derivative along flow} and \eqref{slope inequality}, $u_{m,\delta}$ will be a solution to the perturbed equation \eqref{eq:perturbed equation} which satisfies
\begin{equation}\label{eq:bound of J_delta(u_delta)}
	c_{m,\delta}-a\leq J_{m,\delta}(u_{m,\delta})\leq c_{m,\delta}+a.
\end{equation}

We now wish to take the limits as $a\rightarrow 0$ and $\delta\rightarrow 0$. By the uniform bounds \eqref{large x bounds} and \eqref{small x bounds} from Step 1, we can find $K > 0$ such that
\begin{align*}
	\psi^n_{m,\delta}(z,x) =\eta_\delta\psi^n_m+\delta^2\leq\lambda_1^n(1-\theta)|x|^n+K e^{\sigma \abs{x}^{1+\frac{1}{n}}} +1\quad \text{ for all }x \leq 0, z\in\ov{\Omega}.
\end{align*}
Thus, by the uniform estimate \eqref{key uniform estimate} and the Moser--Trudinger inequality \eqref{eq:MT_ineq}, we have for any $\gamma > 0$ such that $\sigma(1+\gamma) < 2n$, that:
\begin{align*}
	&\int_\Omega\psi_{m,\delta}^{n(1+\gamma)}(z,u_{m,\delta})\omega^n \\
	\leq& \lambda_1^{n(1+\gamma)}(1-\theta)^{1+\gamma}\int_{\Omega}|u_{m,\delta}|^{n(1+\gamma)}\omega^n
	+K\int_\Omega e^{\sigma (1+\gamma) |u_{m,\delta}|^{1+\frac{1}{n}}} \omega^n+\Vol(\Omega) \leq K.
\end{align*}
Hence, by Ko{\l}odziej's $L^\infty$-estimate  \cite{Kol98} (c.f. \cite[Theorem 1.1]{HZ26}), we obtain a uniform $L^\infty$-bound for $\|u_{m,\delta}\|_{C^0}$. By Theorem \ref{thm:C3alpha-estimate}, $u_{m,\delta}$ is thus uniformly bounded in $C^2(\ov{\Omega})$, and we can take the limit as both $a, \delta$ go to zero to produce $u_m\in C^{1,1}(\ov{\Omega})\cap \cH(\Omega)$ solving \eqref{eq:initial equation} with $\psi_m$ and $J_m(u_m) = c_m > \frac{\theta k_m^{n+1}}{n+1}$; observe that this implies that $u_m\not= 0$.

Finally, if we now choose a fixed $m$ which is larger that the uniform $L^\infty$-bound for $u_m$, we will have $\psi_m(z, u_m) = \psi(z, u_m)$, and that $u_m$ will be a non-trivial solution to the original, untruncated problem \eqref{eq:initial equation}.
\end{proof}

\appendix

\section{Construction of $\mu$}\label{mu appendix}

\begin{lemma}\label{construction of mu}
For each $p \in (2, \infty)$, there exists a smooth function $\mu:(0,\infty)\to\mathbb{R}$ such that
\begin{itemize}\setlength{\itemsep}{1mm}
\item $\mu(t)=\log t$ for $0<t\leq1$;
\item $\mu(t)=t^{1/p}$ for $t\geq e^{2}$;
\item $\mu'(t)>0$;
\item $t^{2}\mu''(t)+t\mu'(t) \leq \frac{t}{p}\mu'(t)$.
\end{itemize}
\end{lemma}

\begin{proof}
It suffices to construct a smooth function $\tau:\mathbb{R}\to\mathbb{R}$ such that
\begin{itemize}\setlength{\itemsep}{1mm}
\item $\tau(s)=s$ for $s\leq0$;
\item $\tau(s)=e^{s/p}$ for $s\geq2$;
\item $\tau'(s)>0$; and
\item $\tau''(s)-\frac{1}{p}\tau'(s) \leq 0$.
\end{itemize}
Given such $\tau$, then $\mu(t):=\tau(\log t)$ is the required function. To construct $\tau$, it suffices to construct a smooth function $\vp:\mathbb{R}\to\mathbb{R}$ such that
\begin{enumerate}[i)]\setlength{\itemsep}{1mm}
\item $\vp(s)=e^{-s/p}$ for $s\leq0$;
\item $\vp(s)=1/p$ for $s\geq2$;
\item $\vp'(s)\leq0$; and
\item $\int_{0}^{2}e^{x/p}\vp(x)dx=e^{2/p}$.
\end{enumerate}
Then it is straightforward to verify that $\tau(s):=\int_{0}^{s}e^{x/p}\vp(x)dx$ satisfies the above requirements.

\medskip

We will now explain how to construct $\vp$. It is clear how to construct functions satisfying i)-iii) by using, e.g. a symmetric mollifier; we thus focus on condition iv). Start by choosing a constant $\delta > 0$ sufficiently small such that $(1 - \delta) p > 2$. Then, by the elementary inequality $e^{x}\geq1+x$ for $x\geq0$, we have
\[
(p-1)e^{2/p} \geq (p-1)\left(1+\frac{2}{p}\right) = p+1-\frac{2}{p} > p + \delta.
\]
It follows that:
\[
\int_0^2 e^{x/p} dx = p e^{2/p} - p > e^{2/p} + \delta.
\]
Now, construct a smooth function $\vp_0$ which satisfies i), iii), and $\vp_0(s) = 1 - \e$ for all $s \geq 2\e$, where $\e > 0$ is sufficiently small such that:
\[
\int_0^2 e^{x/p} \vp_0(x) dx \geq \int_0^2 e^{x/p} dx - \delta > e^{2/p},
\]
by our choice of $\delta$.

By using piecewise linear interpolation between $\vp_0$ and the constant function $\frac{1}{p}$, and taking suitable mollifications, we can then construct a continuous family of smooth functions $\vp_{t}$, $t\in (2\e, 2)$, each of which satisfy i) - iii), and such that
\[
\int_0^2 e^{s/p}\vp_t(x)dx \rightarrow \int_0^2 e^{2/p} \vp_0(x)dx \quad\quad\text{ as } t\rightarrow 2\e
\]
and
\[
\int_0^2 e^{s/p}\vp_t(x)dx \rightarrow \int_0^2 \frac{1}{p} e^{2/p}dx = e^{2/p}-1 \quad\quad\text{ as } t\rightarrow 2.
\]
Since $\int_0^2 e^{2/p}\vp_0(x) dx > e^{2/p} > \int_0^2 \frac{1}{p}e^{2/p}dx$, we conclude by continuity that there exists some $t\in (2\e, 2)$ such that $\vp_t$ satisfies iv), as desired.
\end{proof}
\begin{remark}
    By the above construction,
    \begin{align*}
        \tau'(s)=e^{s/p}\varphi(s)
        \geq \frac{1}{p}e^{s/p}.
    \end{align*}
    Thus,
    \begin{align*}
        \mu'(t)=\tau'(s)t^{-1}\geq \frac{1}{p}t^{\frac{1}{p}-1},
    \end{align*}
    which is equivalent to $\mu(t)-t^{1/p}$ being nondecreasing. Thus, we have
    \begin{align}\label{slope inequality}
        (t_1-t_2)(\mu(t_1)-\mu(t_2))
        \geq (t_1-t_2)(t_1^{\frac{1}{p}}-t_2^{\frac{1}{p}})\quad \text{ for all }\ t_1,t_2>0
    \end{align}
\end{remark}

\bigskip

Denote the positive orthant of $\mathbb{R}^{n}$ by
\[
\mathcal{P} = \big\{\lambda=(\lambda_{1},\ldots,\lambda_{n})\in\mathbb{R}^{n}:\lambda_{i}>0 \ \text{for each $i$}\big\}.
\]
Define
\[
F(\lambda) = \mu\big(\MA(\lambda)\big) \ \text{where} \ \MA(\lambda) = \prod_{i}\lambda_{i}.
\]
\begin{lemma}\label{concavity of F}
When $p\geq n$, the function $F$ is concave on $\mathcal{P}$.
\end{lemma}

\begin{proof}
For notational convenience, we write
\[
\mu' = \mu'\big(\MA(\lambda)\big), \ \
\mu'' = \mu''\big(\MA(\lambda)\big), \ \
\MA = \MA(\lambda).
\]
We compute
\[
\frac{\de F}{\de\lambda_{i}} = \mu'\cdot\frac{\MA}{\lambda_{i}}
\]
and
\[
\frac{\de^{2}F}{\de\lambda_{i}\de\lambda_{j}} = (\mu''\cdot\MA^{2}+\mu'\cdot\MA)\cdot\frac{1}{\lambda_{i}\lambda_{j}}-\mu'\cdot\MA\cdot\frac{\delta_{ij}}{\lambda_{i}\lambda_{j}}
\]
Since $t^{2}\mu''(t)+t\mu'(t) \leq \frac{t}{p}\mu'(t)\leq \frac{t}{n}\mu'(t)$, then for any $(\xi_{1},\ldots,\xi_{n})\in\mathbb{R}^{n}$, we have
\[
\sum_{i,j}\frac{\de^{2}F}{\de\lambda_{i}\de\lambda_{j}}\xi_{i}\xi_{j}
\leq \mu'\cdot\MA\cdot\frac{1}{n}\left(\sum_{i}\frac{\xi_{i}}{\lambda_{i}}\right)^{2}-\mu'\cdot \MA\cdot \sum_{i}\frac{\xi_{i}^{2}}{\lambda_{i}^{2}}.
\]
Combining this with the Cauchy-Schwarz inequality
\[
\left(\sum_{i}\frac{\xi_{i}}{\lambda_{i}}\right)^{2} \leq n\sum_{i}\frac{\xi_{i}^{2}}{\lambda_{i}^{2}},
\]
we are done.
\end{proof}

	\bibliographystyle{alpha}
	\bibliography{Reference}
	
\end{document}